\documentclass[a4paper,11pt]{amsart}
\usepackage{amsmath,amssymb,amsthm}
\usepackage{tikz}
\usepackage{hyperref}
\usepackage[all]{xy}

\title{The ideal of maximal flags of a poset}
\author{Amin Nematbakhsh}
\date{\today}
\email{nematbakhsh.amin@gmail.com}

\newtheorem{theorem}{Theorem}[section]
\newtheorem{proposition}[theorem]{Proposition}
\newtheorem{lemma}[theorem]{Lemma}

\newtheorem{claim}[]{Claim}
\theoremstyle{definition}
\newtheorem{definition}[theorem]{Definition}
\newtheorem{remark}[theorem]{Remark}

\newtheorem{example}[theorem]{Example}

\newtheorem{construction}[theorem]{Construction}

\newtheorem{lemma*}[]{Lemma}

\newcommand{\kk}{\Bbbk}
\newcommand{\NN}{\mathbb{N}}
\newcommand{\bfa}{\mathbf{a}}
\newcommand{\bfb}{\mathbf{b}}
\newcommand{\bfc}{\mathbf{c}}

\newcommand{\mfp}{\mathfrak{p}}
\newcommand{\mcF}{\mathcal{F}}
\newcommand{\mcI}{\mathcal{I}}
\newcommand{\mcJ}{\mathcal{J}}
\newcommand{\mcK}{\mathcal{K}}
\newcommand{\mcM}{\mathcal{M}}

\newcommand{\ass}{\operatorname{ass}}
\newcommand{\height}{\operatorname{height}}

\newcommand{\Hom}{\operatorname{Hom}}
\newcommand{\Tor}{\operatorname{Tor}}

\begin{document}

\begin{abstract}
We study algebraic and homological properties of facet ideals of order complexes of posets which we call ideals of maximal flags of posets or simply flag ideals.
We characterize the unmixed and Cohen-Macaulay flag ideals of graded posets.
We also give structural results for the multigraded Betti numbers of flag ideals of such posets.
The structural results on multigraded Betti numbers of flag ideals are used to characterize the class of flag ideals with linear resolutions.
%Moreover, we show that the same characterizations are valid for the edge ideals of $r$-partite $r$-uniform hypergraphs.
\end{abstract}

\maketitle

\section{Introduction}

Let $V$ be a finite set and let $R=\kk[x_V]$ be a polynomial ring with indeterminates $x_i$, for all $i\in V$.
There are two ways to assign to any simplicial complex $\Delta$ on vertex set $V$, an ideal in the polynomial ring $R$.
The more common way is to construct the Stanley-Resiner ideal $I_\Delta$, defined as
\[
I_\Delta = <\prod_{i\in N} x_i ~|~ N\in 2^V \setminus \Delta>.
\]
The other way is by constructing the facet ideal of $\Delta$, denoted by $I(\Delta)$ and defined to be
\[
I(\Delta) = <\prod_{i\in F} x_i ~|~ F \text{ is a facet of } \Delta>.
\]
In this manuscript, we study the facet ideals of simplicial complexes $\Delta(P)$ where $\Delta(P)$ is the order complex of a finite poset $P$.
% as well as the edge ideals of $r$-partite $r$-uniform hypergraphs.
% such that $P$ admits a rank function.
We call the facet ideal of $\Delta(P)$ the flag ideal of $P$ and denote it by $\mcF(P)$.

\medskip
The class of such ideals includes the class of bipartite graphs and letterplace ideals.
The algebraic properties of bipartite graphs has been studied by several authors in the literature.
Some of the known results include:
characterization of Cohen-Macaulay bipartite graphs \cite{Herzog-Hibi-02},
characterization of unmixed bipartite graphs \cite{Villarreal-01} and
characterization of bipartite graphs with linear resolutions \cite{Corso-Nagel-01}.
The letterplace ideals were first appeared in \cite{Ene-Herzog-Mohammadi-01}. In the same article it is shown that the Alexander dual of letterplace ideals have linear quotients which implies their Cohen-Macaulayness.
In \cite{Floystad-Greve-Herzog-01}, it is shown that several classes of monomial ideals are regular quotients of letterplace ideals.
The multigraded Betti numbers of letterplace ideals are studied in \cite{DAli-Floystad-Nematbakhsh-01}.

\medskip
The Stanley-Reisner ideal of simplicial complexes are extensively studied by several authors.
In \cite{Reisner-01}, G.~A.~Reisner gave a complete characterization of Cohen-Macaulay Stanley-Reisner ideals via the homology of links of the corresponding simplicial complexes.
The relationship between a simplicial complex and its facet ideal is less studied in the literature.
%The facet ideal of simplicial complexes are less studied in the literature.
Algebraic properties of the facet ideal of a special class of simplicial complexes called the simplicial trees are studied in \cite{Faridi-01,Faridi-02}.
In this paper, we study flag ideals of posets that are equipped with a rank function.
We provide combinatorial characterizations of unmixed and Cohen-Macaulay flag ideals based on the structure of the poset.
%In this article, we show that the facet ideals of simplicial complexes arising from posets that admit a rank function have very nice algebraic and homological behavior.
Let $P$ be a finite graded poset.
We denote the induced subposet on the elements of rank $i$ and $i+1$ with $P_{i,i+1}$.
Let $Q_{i,i+1}$ be the subposet of $P_{i,i+1}$ after removing the rank $1$ maximal elements of $P_{i,i+1}$.
The Hasse diagram of the subposet $Q_{i,i+1}$ is a bipartite graph.
The flag ideal of $Q_{i,i+1}$ is exactly the edge ideal of the Hasse diagram of $Q_{i,i+1}$ when considered as a bipartite graph.
We study the relationship between the unmixedness and the Cohen-Macaulayness of the flag ideal $\mcF(P)$ with the same property of the bipartite graphs $Q_{i,i+1}$.
%We use algebraic properties of the bipartite graphs $Q_{i,i+1}$ to investigate algebraic properties of the flag ideal of $P$.
This relationship enables us to give combinatorial characterizations for unmixed and Cohen-Macaulay flag ideals generalizing characterization of unmixed and Cohen-Macaulay bipartite graphs given in \cite{Villarreal-01} and \cite{Herzog-Hibi-02} respectively.
Moreover, we show that a flag ideal has a linear resolution if and only if the edge ideals of bipartite graphs $Q_{i,i+1}$ have linear resolutions.
%We also show that the same characterizations for unmixed, Cohen-Macaulay and flag ideals with linear resolutions are also valid for the case of edge ideals $r$-partite $r$-uniform hypergraphs.

%\medskip
%Let $H$ be an $r$-partite hypergraph and let $A_1\cup\cdots\cup A_r$ be the partition of vertices of $H$.
%After we completed the work on flag ideals, we noticed that the proofs are also valid in the case of edge ideals of $r$-partite $r$-uniform hypergraphs.

%\begin{enumerate}
%\item[(*)] Any edge of the hypergraph $H$ is of the form $\{a_1,\ldots,a_i\}$ such that for $j=1,\ldots,i$, $a_j\in A_j$.
%\end{enumerate}
%Even though that the results for flag ideals hold in this setting but we restrict ourselves to the more interesting case of $r$-partite $r$-uniform hypergraphs.
%The last section is devoted to the study of edge ideals of such hypergraphs.

\medskip
{\it Unmixedness and Cohen-Macaulayness.}
In \cite{Herzog-Hibi-02}, J.~Herzog and T.~Hibi completely characterized the Cohen-Macaulay edge ideals of  bipartite graphs.
They show that the edge ideal of a bipartite graph is Cohen-Macaulay if and only if it is the quadratic letterplace ideal of a poset.
Inspired by the work of Herzog and Hibi, R.~Villarreal gave a combinatorial characterization for unmixedness of edge ideals of bipartite graphs.
A generalization of Villarreal's result for unmixedness to $d$-uniform $r$-partite hypergraphs is given in \cite{JafarpourGolzari-ZaareNahandi-01} provided that the hypergraph satisfies a certain condition.
Here we show that a similar characterization also holds for unmixed and Cohen-Macaulay flag ideal of a graded poset $P$.
%Another natural generalization of bipartite graphs are the $r$-partite $r$-uniform hypergraphs.
%We provide characterizations for unmixed and Cohen-Macaulay edge ideals of $r$-partite $r$-uniform hypergraphs as well.

\medskip
{\it Linear resolutions.}
A.~Corso and U.~Nagel characterized the class of edge ideals of bipartite graphs with linear resolutions \cite{Corso-Nagel-01}.
They show that the edge ideal of a bipartite graph has a linear resolution if and only if it is a Ferrers graph.
They explicitly described the minimal free resolution of Ferrers graphs associated to a partition $\lambda$ and computed the Betti numbers.
They generalized their results to the class of edge ideals of threshold graphs by a process called specialization in \cite{Corso-Nagel-02}.
In \cite{Nagel-Reiner-01}, U.~Nagel and V.~Reiner introduced the class of Ferrers hypergraphs as a generalization of Ferrers graphs.
They show that the edge ideal of Ferrers hypergraphs along with the edge ideal of hypergraphs associated to (squarefree) strongly stable $r$-uniform hypergraphs have linear resolutions.
A minimal cellular free resolution is constructed in this case as well.
%In Corollary \ref{cor-linearHyper}, we give another proof for the fact that the edge ideals of such hypergraphs have linear resolutions.
The edge ideals of such graphs are not in general flag ideals.
We give a characterization for flag ideals with linear resolutions in Theorem \ref{thm-minres}.
We show that a flag ideal $\mcF(P)$ has a linear resolution if and only if the the poset is pure and the edge ideal of biparite graphs built on elements of consecutive ranks of the poset have linear resolutions.

\medskip
{\it Organization of the paper.}
In section \ref{sec-flagIdeals}, we provide the basic definitions and terminology required for the rest of the paper.
We also give some examples of flag ideals.
In sections \ref{sec-unmixed} and \ref{sec-CM}, we characterize unmixed and Cohen-Macaulay flag ideals within the class of graded posets.
We show that, if a flag ideal is unmixed (resp. Cohen-Macaulay) then the consecutive bipartite graphs of its Hasse diagram are unmixed (resp. Cohen-Macaulay) as well but not conversely.
Section \ref{sec-multidegrees} investigates structure of the multigraded Betti numbers of flag ideals.
We show that the restriction of the Stanley-Reisner complex of a flag ideal to a multidegree is the join of simplicial complexes associated to certain bipartite graphs.
A characterization for multidegrees of the first linear strand of such flag ideals is also given in this section.
In section \ref{sec-linearRes}, we show that a bipartite graph is a Ferrers graph if and only if it does not have an induced subgraph consisting of two disjoint edges.
This description of Ferrers graphs is used to provide a characterization of flag ideals with linear resolutions.

\medskip
{\it Acknowledgement.} The author would like to thank Gunnar Fl{\o}ystad for the reading of the manuscript and his valuable comments. 
 
\section{Ideals of maximal flags}
\label{sec-flagIdeals}

Let $P$ be a finite poset. A {\it flag} in $P$ is a chain of elements
\[
p_1 < p_2 <\cdots < p_n
\]
of $P$. Let $R=\kk[x_P]$ be a polynomial ring in indeterminates $x_i$ for all $i\in P$.
For any subset $I\subseteq P$ we denote the square-free monomial $\prod_{i\in I} x_i$ by $x_I$.
We define the {\it ideal of maximal flags} (or simply the {\it flag ideal}) $\mcF(P)$ of $P$ to be the monomial ideal generated by
all monomials $x_I$ where $I=\{p_1 ,\cdots , p_n\} \subseteq P$ and $p_1<\cdots<p_n$ is a maximal chain in $P$.

The {\it order complex} of a poset $P$ is a simplicial complex on the vertex set $P$ and the set of chains in $P$ as faces. We denote the order complex of a poset $P$ by $\Delta(P)$.
Let $\Delta$ be a simplicial complex on the vertex set $V$. The {\it facet ideal} $\mathfrak{F}(\Delta)$ of $\Delta$ is a square-free monomial ideal in the polynomial ring $\kk[x_{V}]$ generated by monomials $x_{v_1}\cdots x_{v_n}$ where $\{v_1,\ldots,v_n\}$ is a facet of $\Delta$.
The flag ideal of $P$ is the facet ideal of the order complex of $P$.

Let $P$ be a poset. A {\it rank function} on $P$ is a function $r_P:P\to \NN$ satisfying the following two conditions
\begin{enumerate}
\item $r_P(p) = 1$ if $p$ is a minimal element of $P$;
\item if $q$ covers $p$ then $r_P(q) = r_P(p) +1$.
\end{enumerate}
A poset is called {\it graded} if it has a rank function.
Not all posets admit a rank function.

\begin{example}
The {\it pentagon poset}
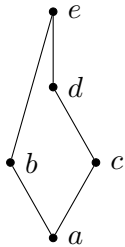
\begin{center}
 \begin{tikzpicture}[scale=.75, vertices/.style={draw, fill=black, circle, inner sep=1pt}]
              \node [vertices, label=right:{$a$}] (0) at (-0+0,0){};
              \node [vertices, label=right:{$b$}] (1) at (-.75+0,1.33333){};
              \node [vertices, label=right:{$c$}] (3) at (-.75+1.5,1.33333){};
              \node [vertices, label=right:{$d$}] (4) at (-0+0,2.66667){};
              \node [vertices, label=right:{$e$}] (2) at (-0+0,4){};
      \foreach \to/\from in {0/1, 0/3, 1/2, 3/4, 4/2}
      \draw [-] (\to)--(\from);
      \end{tikzpicture}
\end{center}
does not have a rank function.
\end{example}

Suppose $P$ is graded with the rank function $r_P$.
For an element $p\in P$, $r_P(p)$ is called the rank of $p$.
The number
\[\max \{r_P(p)~|~p\in P\}\]
is called the {\it rank} of $P$ and we denote it by $\overline{r}(P)$.
%This is the length of a longest chain in $P$ plus 1.
We also define $\underline{r}(P)$ to be
\[
\min\{ r_P(p)~|~p\in \max(P)\}
\]
where $\max(P)$ is the set of maximal elements of $P$.
A poset is called {\it pure} if all maximal elements of $P$ have the same rank, or equivalently, all maximal chains have the same length.
In the poset $P$, if an element $q$ covers an element $p$, we say $p$ is a {\it parent} of $q$ and we also say $q$ is a {\it child} of $p$.

\begin{example}
Some classes of ideals that are special cases of flag ideals are as follows.
%\begin{enumerate}

\medskip
{\it Letterplace ideals.} Let $Q$ be a finite poset and let $n$ be an integer.
The $n$-th letterplace ideal $L(n,Q)$ is the monomial ideal in $\kk[x_{[n]\times Q}]$ generated by all monomials $x_{(1,q_1)}\cdots x_{(n,q_n)}$ where $q_1<\cdots <q_n$ is a multichain in $Q$.
Define the poset $P$ on the ground set $[n]\times Q$, with cover relations 
\[
(i,q) \leq (i+1,q') \text{ for all } 1\leq i \leq n-1 \text{ and for all } q \leq q' \text{ in } Q.
\]
It is not hard to show that $\mcF(P) = L(n,Q)$.

\medskip
{\it Edge ideals of bipartite graphs.} A bipartite graph $G$ is a simple graph where the vertex set of $G$ can be written as a union $A\cup B$, such that any edge of $G$ contains an element of $A$ and an element of $B$.
The edge ideal of $G$ is the monomial ideal $I(G)$ in the polynomial ring $\kk[x_G]$ generated by all monomials $x_ax_b$ where $\{a,b\}$ is an edge of $G$.
Let $P$ be the poset with cover relations 
\[
a \leq b \text{ where }a\in A, b\in B \text{ and } \{a,b\} \text{ is an edge of } G.
\]
We can see that $\mcF(P) = I(G)$.
We usually make no difference between the bipartite graph $G$ and the poset $P$.

\medskip
{\it Co-letterplace ideals of $V$ posets.} Let $Q$ be a finite poset. For an integer $n$, the $n$-th co-letterplace ideal $L(Q,n)$ is a square-free monomial ideal generated by monomials $x_{\Gamma\phi}$ in the polynomial ring $\kk[x_{Q\times [n]}]$ where
$\phi:Q\to [n]$ is an isotone map and
\[
\Gamma\phi = \{(q,\phi(q))~|~ q\in Q\}
\]
is the graph of $\phi$.
Let $Q$ be a poset with exactly two maximal chains and a unique minimal element.
We call such a poset, a {\it $V$ poset}, since its Hasse diagram has a $V$-like shape.
One can show that $L(Q,n)$ is a flag ideal.
Let $\{a<b_1<\cdots<b_r\}$ and $\{a<c_1<\cdots<c_s\}$ be the two maximal chains of $Q$.
Let $P$ be the poset on the ground set $Q\times [n]$ with covering relations
\begin{align*}
&(a,i) < (b_1,j)  \text{ if } i \leq j,\\
&(b_k,i) < (b_{k+1},j)  \text{ if } i \leq j, \text{ for } k=1,\ldots, r-1,\\
&(c_1,i) < (a,j) \text{ if } j \leq i,\\
&(c_{k},i) < (c_{k-1},j) \text{ if } j \leq i, \text{ for } k=2,\ldots,s.
\end{align*}
The flag ideal of $P$ is exactly the co-letterplace ideal $L(Q,n)$.
%\end{enumerate}
\end{example}

\begin{definition}
\label{def-VertexCover}
Let $P$ be a finite poset. We call a subset $C$ of $P$ a {\it vertex cover} if the intersection of $C$ with any maximal chain in $P$ is nonempty. A {\it minimal vertex cover} of $P$ is a vertex cover $C$ such that no proper subset of $C$ is a vertex cover of $P$. The smallest cardinality among the minimal vertex covers of $P$ is called the {\it vertex covering number} of $P$.
The poset $P$ is called {\it unmixed} if all of its minimal vertex covers have the same cardinality.
These notions coincide with the corresponding notions defined in \cite{Faridi-01} for a simplicial complex $\Delta$, when $\Delta = \Delta(P)$.
\end{definition}

A subset $C = \{p_1,\ldots,p_t\}$ is a vertex cover of $P$ if and only if for any minimal generator $g$ of $\mcF(P)$, $x_{p_i}| g$, for some $1\leq i\leq t$.
It follows that a prime ideal $(x_{p_1},\ldots,x_{p_t})$ is a minimal prime ideal of $\mcF(P)$ if and only if $\{p_1,\ldots,p_t\}$ is a minimal vertex cover of $P$.
Therefore, the Alexander dual is
\[
\mcF(P)^A = \{\prod_{p\in C}x_p | C \text{ is a minimal vertex cover of } P\}.
\]
\noindent
Furthermore,
\[
\dim R/\mcF(P) = |P| - \text{ vertex covering number of $P$},
\]
and
\[
\height \mcF(P) = \text{ vertex covering number of $P$}.
\]

Let $P$ be a finite graded poset of rank $r$ with a rank function $r_P$.
Let $[r]=\{1<\cdots<r\}$.
For any subset $S\subseteq [r]$ let $P_S$ be the induced subposet of $P$ on the set
\[
P_S = \{p\in P|r_P(p) \in S\}
\]
Let $S=\{r_1,\cdots, r_s\}$ where $1\leq r_1< \cdots<r_s\leq r$.
We call $\mathcal{F}(P_S)$ the {\it partial flag ideal} with respect to $S$ and denote it by $\mcF(P,S)$ or $\mcF(P,\{r_1,\cdots,r_s\})$.
In particular, when $S=\{s\}$ consists of a single element then $P_s$ denotes the set of elements of rank $s$ in $P$.

%\begin{definition}
%Let $P$ and $Q$ be two finite ranked posets on disjoint underlying sets with rank functions $r_P$ and $r_Q$ respectively.
%Suppose $|P_{\overline{r}(P)}| = |Q_1|$, i.e. the set of elements of maximal rank of $P$ is in one-to-one correspondence with the set of elemenst of minimal rank of $Q$. Let $\phi: P_{\overline{r}(P)} \to Q_1$ be such correspondence.
%We define the {\it tower product} of $P$ and $Q$ with respect to $\phi$ denoted by $P \vDash_\phi Q$ to be a poset on the underlying set $(P\cup Q) \backslash Q_1$ and with relations defined as follows.
%For $x,y\in P\vDash_\phi Q$, $x\leq y$ if and only if
%\begin{enumerate}
%\item $x,y\in P$ and $x \leq y$ in $P$;
%\item $x,y\in Q$ and $x \leq y$ in $Q$;
%item $x\in P$, $y\in Q$ and there exist some $z\in P$ such that $x\leq z$ in $P$ and $\phi(z) \leq y$ in %$Q$.
%\end{enumerate}
%The Hasse diagram of $P\vDash_\phi Q$ is obtained by putting the Hasse diagram of $Q$ over the Hasse diagram of $P$ and glue them via the correspondence map $\phi$. Note that the tower product $P\vDash_\phi Q$ indeed depends on the correspondence map $\phi$. The tower product is associative but obviously it is not commutative in general.
%\end{definition}

%\subsection{Reduction to connected components}

Let $P$ be a finite poset.
We say that two elements $a$ and $b$ of $P$ are {\it connected} if there is a sequence $a=p_0,\ldots,p_n=b$ of elements of $P$ such that for $i=0,\ldots,n-1$, either $p_i \leq p_{i+1}$ or $p_{i+1}\leq p_i$.
The being connected relation is an equivalence relation. We call equivalence classes of this relation, {connected components} of $P$.
These are exactly the connected components of the Hasse diagram of $P$ when viewed as a directed graph.

%In this subsection we show that the properties of being unmixed, Cohen-Macaulay and having a linear resolution can be reduced to connected components of $P$. That is $\mcF(P)$ is unmixed, Cohen-Macaulay or has a linear resolution if and only if for each connected component $Q$ of $P$, $\mcF(Q)$ is respectively unmixed, Cohen-Macaulay or has a linear resolution.

\begin{remark}
In \cite{Faridi-01}, S.~Faridi introduced the notion of a simplicial tree, generalizing the notion of trees in graph theory to simplicial complexes. She studied the facet ideal of such complexes.
A tree is a bipartite graph. Hence regarded as a 1-dimensional simplicial complex, its facet ideal is the flag ideal of a rank 2 poset. But the facet ideal of a simplicial tree is not in general a flag ideal.

As an example consider the simplicial complex 
\[
\Delta = <\{x_1,x_2,x_3\},\{x_1,x_4,x_5\},\{x_2,x_6,x_7\},\{x_3,x_8,x_9\}>.
\]
There is no poset $P$ on vertices $x_1,\ldots,x_9$ such that the flag ideal of $P$ gives the facet ideal of $\Delta$.
\end{remark}

\section{Unmixed flag ideals}
\label{sec-unmixed}

I this section we characterize the class of unmixed graded posets.
The main theorem of this section, Theorem \ref{thm-unmixedness}, generalizes the result of R.~Villarreal on unmixed bipartite graphs to the class of graded posets.

\begin{definition}
An ideal $I$ of a Noetherian ring $R$ is called unmixed if the height (or codimension) of $I$ is equal to the height of every associated prime $\mfp$ of $R/I$.
\end{definition}

There is a different notion of unmixedness used in the literature that defines an ideal $I$ to be unmixed if it has no embedded associated prime ideal.
The former notion of unmixedness implies the latter but not conversely. For example squarefree monomial ideals in a polynomial ring are always unmixed in the latter sense.

\medskip
The flag ideal of a poset $P$ is a square-free monomial ideal and $\ass(\mcF(P)) = \min(\ass(\mcF(P)))$.
The minimal prime ideals of $\mcF(P)$ are in one-to-one correspondence with minimal vertex covers of $P$ as explained by the discussion after Definition \ref{def-VertexCover}.
Therefore, $\mcF(P)$ is unmixed if and only if all minimal vertex covers of $P$ have the same cardinality, i.e. if and only if $P$ is an unmixed poset.
In particular, a poset $P$ is unmixed if and only if all of its connected components are unmixed.

%\begin{lemma}
%Let $P$ be a finite poset with a rank function.
%If $P$ is unmixed then it is ranked and for $1\leq i\leq \overline{r}(P)$, all $P_i$ have the same cardinality.
%In particular, $\overline{r}(P) | |P|$.
%\end{lemma}

%\begin{proof}
%Let $r=\overline{r}(P)$ be the rank of $P$.
%Clearly, if $P$ is unmixed then for $1\leq i\leq$, $P_i$ forms a minimal vertex cover. Hence they all have same cardinality.
%Suppose $P$ is not ranked and $s=\underline{r}(P)$.
%Then $P_{s+1} \cup (P_{s} \backslash \max(P))$ is a minimal vertex cover of $P$.
%\end{proof}

\begin{lemma}[Rank selection for unmixedness] \label{pro-rankselectionunmixed}
Let $P$ be a finite unmixed poset of rank $r$.
% that admits a rank function $r_P$.
%Let $r=r(P)$ be the rank of $P$.
\begin{enumerate}
\item If $P$ is graded then for any subset $S = \{i,i+1,\ldots,i+j\}$ of $[r]$, $\mcF(P,S)$ is unmixed;
\item If $P$ is pure then for any subset $S\subseteq [r]$, $\mcF(P,S)$ is unmixed.
\end{enumerate}
%If $P$ is unmixed then for any $S\subseteq [r]$, $P_S$ is also unmixed. In particular $P$ is a tower product of unmixed bipartite graphs.
\end{lemma}

\begin{proof}
%Let $r$ be the rank of $P$. 
%For all $1\leq i\leq r$, the subsets $P_i$ have same cardinality, since they form minimal vertex covers of $P$.
$(1)$ Suppose $P$ is unmixed and admits a rank function.
The set of elements of minimal rank $P_1$ is a minimal vertex cover of $P$.
Let $|P_1|=t$. Since $P$ is unmixed, any minimal vertex cover of $P$ is of cardinality $t$.
Suppose $C$ is a minimal vertex cover of $P_S$.
Let $i=\min S$, be the minimal integer in $S$ and 
let
\[
D=\{p\in \max P ~|~ r_P(p) < i\}.
\]
We claim that $C\cup D$ is a minimal vertex cover of $P$.
Let $\bfc = \{p_1< \cdots < p_n\}$ be a maximal chain in $P$.
%Any maximal chain in $P$ gives a maximal chain in $P_S$, hence it intersects $C$.
If $r_P(p_n) < i$ then $p_n\in D$.
If $r_P(p_n) \geq i$ then $\bfc$ defines a maximal chain in $P_S$, hence it intersects $C$.
In either case $\bfc$ intersects $C\cup D$.
Therefore, $C\cup D$ is a vertex cover of $P$ and it contains a minimal vertex cover $E$.
Note that if we remove any point of $D$ from $C\cup D$ then it is no longer a minimal vertex cover. 
If $E \neq C\cup D$ then $E = C' \cup D$ where $C'$ is a proper subset of $C$.
Since $C$ is a minimal vertex cover, there is a maximal chain in $P_S$ that does not intersect $C'$. Any such chain can be extended to a maximal chain of $P$ that does not intersect $E$, which is a contradiction. Therefore any minimal vertex cover $C$ of $P_S$ gives a minimal vertex cover
\[C\cup \{p\in \max P ~|~ r_P(p) < i\}\]
of $P$. Since $P$ is unmixed, all of them must have same cardinality.

\medskip
$(2)$ Proof of this part is similar.
\end{proof}

\begin{example}
\label{exp-NotUnmixed}
Consider the poset $P$ with Hasse diagram
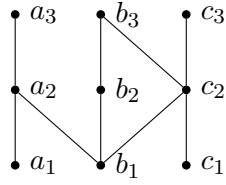
\begin{center}
 \begin{tikzpicture}[scale=.75, vertices/.style={draw, fill=black, circle, inner sep=1pt}]
             \node [vertices, label=right:{${a}_{1}$}] (0) at (-1.5+0,0){};
             \node [vertices, label=right:{${b}_{1}$}] (2) at (-1.5+1.5,0){};
             \node [vertices, label=right:{${c}_{1}$}] (5) at (-1.5+3,0){};
             \node [vertices, label=right:{${a}_{2}$}] (1) at (-1.5+0,1.33333){};
             \node [vertices, label=right:{${b}_{2}$}] (3) at (-1.5+1.5,1.33333){};
             \node [vertices, label=right:{${c}_{2}$}] (4) at (-1.5+3,1.33333){};
             \node [vertices, label=right:{${a}_{3}$}] (6) at (-1.5+0,2.66667){};
             \node [vertices, label=right:{${b}_{3}$}] (8) at (-1.5+1.5,2.66667){};
             \node [vertices, label=right:{${c}_{3}$}] (7) at (-1.5+3,2.66667){};
     \foreach \to/\from in {0/1, 1/6, 2/4, 2/1, 2/3, 3/8, 4/8, 4/7, 5/4}
     \draw [-] (\to)--(\from);
     \end{tikzpicture}
\end{center}
Both of the subposets $P_{1,2}$ an $P_{2,3}$ are unmixed but $P$ itself is not unmixed.
The subset $\{a_1,b_1,b_3,c_3\}$ is a minimal vertex cover of cardinality 4 while all the other minimal vertex covers have cardinality 3.
\end{example}

\begin{definition}
Let $P$ be a finite poset.
We call a subset $A$ of elements of $P$ an {\it independent set} if for every maximal chain $\bfc=\{p_1 <\cdots< p_n\}$ of $P$, $\bfc \nsubseteq A$.
A {\it maximal independent set} is an independent set $A$ which is not a pure subset of any other independent set.
\end{definition}

It is easy to show that a subset $A$ is a maximal independent set of a poset $P$ if and only if $C=P \backslash A$ is a minimal vertex cover of $P$.

\begin{theorem}
\label{thm-unmixedness}
Let $P$ be a graded poset of rank $r$.
The ideal $\mcF(P)$ is unmixed if and only if
\begin{enumerate}
%\item for all $1\leq i\leq \overline{r}(P)-1$, $|P_i| = |P_{i+1}|$ and there is a one-to-one correspondenc $\phi : P_i \to P_{i+1}$ such that for all $p\in P_i$, $p \leq \phi(p)$ in $P$;
\item $|P_1| \geq |P_2| \geq \cdots \geq |P_r|$;
\item for each $1\leq i\leq r$ there is an ordering $P_i = \{a^i_1,\ldots,a^i_{t_i}\}$ such that
for each $2\leq i\leq r$, and each $1\leq u\leq t_i$, $a^{i-1}_u \leq a^i_u$ in $P$;
%for each $u\in [t]$, $a^1_u<\cdots< a^r_u$ forms a maximal chain in $P$;
\item for $i<j$ if $\{a^i_u < a^{i+1} < \cdots < a^{j-1} < a^j_v\}$ and $\{a^i_v < b^{i+1} < \cdots < b^{j-1} < a^j_w\}$ are two saturated chains in $P$ (with $a^i_u, a^i_v \in P_i$ and $a^j_v,a^j_w \in P_j$), then there is a saturated chain $\{a^i_u < c^{i+1} < \cdots <c^{j-1} <a^j_w\}$ such that 
\[
\{c^{i+1},\ldots,c^{j-1}\} \subseteq \{a^{i+1} , \cdots , a^{j-1} \} \cup \{b^{i+1} , \cdots , b^{j-1}\}.
\]
That is, all the intermediate elements in the chain 
\[\{a^i_u < c^{i+1} < \cdots < c^{j-1} <a^j_w\}\]
are among the elements of the two given chains.
\item for $i<j<k$, if $\{a^i_u < a^{i+1} < \cdots < a^{k-1} < a^k_v\}$ and $\{a^i_v < b^{i+1} < \cdots < b^{j-1} < a^j_w\}$ are two saturated chains in $P$ (with $a^i_u, a^i_v \in P_i$, $a^k_v \in P_k$ and $a^j_w \in P_j\cap \max(P)$), then there is a saturated chain $\{a^i_u < c^{i+1} < \cdots <c^{j-1} <a^j_w\}$ such that 
\[
\{c^{i+1},\ldots,c^{j-1}\} \subseteq \{a^{i+1} , \cdots , a^{j-1} \} \cup \{b^{i+1} , \cdots , b^{j-1}\}.
\]
%That is all the intermediate elements in the chain $a^i_u < c^{i+1} < \cdots < c^{j-1} <a^j_w$ are among the elements of the two given chains.
\end{enumerate}
\end{theorem}

\begin{proof}
$(\Rightarrow)$ If $\mcF(P)$ is unmixed then by rank selection lemma, Lemma \ref{pro-rankselectionunmixed}, each of the posets $P_{i,i+1}$ are unmixed for $i=1,\ldots,r-1$.
This implies that the induced subgraph of $P_{i,i+1}$ on $(P_i \backslash \max(P)) \cup P_{i+1}$ is an unmixed bipartite graph. Hence $|P_i| \geq |P_{i+1}|$.
Let $B_i = P_i \backslash \max(P)$.
By the discussion in \cite[page 163]{Herzog-Hibi-01}, $|B_i| = |P_{i+1}|$ and we can order elements of $P_i$ and $P_{i+1}$ as $P_i = \{a^i_1,\ldots,a^i_{t_i}\}$ and $P_{i+1} = \{a^{i+1}_1,\ldots,a^{i+1}_{t_{i+1}}\}$ such that for each $1\leq u\leq |B_i|$, there is an edge between $a^{i}_u$ and $a^{i+1}_u$ in $P_{i,i+1}$. This proves $(1)$ and $(2)$.

By $(1)$, there are $t_1$ disjoint maximal chains $\{a^1_j<\ldots<a_j^{r_j}\}$ in $P$, for $j=1,\ldots,t_1$. 
Observe that if $C$ is a minimal vertex cover of $P$ then $|C|=t_1$ and $C$ intersects with each of the maximal chains $\{a^1_j<\cdots<a^{r_j}_j\}$ in exactly one point.
Now let $\{a^i_u < a^{i+1} < \cdots < a^{j-1} < a^j_v\}$ and $\{a^i_v < b^{i+1} < \cdots < b^{j-1} < a^j_w\}$ be two saturated chains in $P$ as in $(3)$. If $\{a^i_u,a^j_w\}\cup \{a^{i+1} , \cdots , a^{j-1} \} \cup \{b^{i+1} , \cdots , b^{j-1}\}$ does not contain a maximal chain in the poset $Q = P_{i,\ldots,j}$ then it is an independent set of $Q$. We extend it to a maximal independent set $A$. Note that $Q$ is also unmixed and the complement of $A$ in the ground set of $Q$ form a minimal vertex cover $C$.
By construction, $a^i_v$ and $a^k_v$ are in $C$, which is in contradiction with the fact that $C$ intersects the chain $\{a^i_v<\cdots<a^k_v\}$ in exactly one point.
Proof of $(4)$ is similar.

\medskip
$(\Leftarrow)$
Let $C$ be a minimal vertex cover of $P$. For each $1\leq v\leq t_1$, $C$ intersects the set $\{a^1_v,\ldots,a^{r_v}_v\}$ in at least one point. It suffices to show that the intersection contains exactly one point.
Suppose for some $v$, $C$ contains the points $a^i_v$ and $a^j_v$ with $i<j$.
For a point $p\in P$, let $\mcM(p)$ be the set of all maximal chains in $P$ containing $p$.
Suppose $p$ is a point in $C$.
If any chain in $\mcM(p)$ intersects $C$ in at least two points, then $C\backslash\{p\}$ is also a vertex cover, which contradicts our assumption that $C$ is a minimal vertex cover.
Therefore, if $p\in C$, then there exists at least a chain in $\mcM(p)$ that intersects $C$ only in the point $p$.
Let $\bfa = \{a^1<\ldots<a^r\}$ be a maximal chain in $\mcM(a^i_v)$ that intersects $C$ only in $a^i_v$.
Similarly let $\bfb=\{b^1<\ldots<b^s\}$ be a maximal chain in $\mcM(a^j_v)$ that intersects $C$ in only one point.

Consider the chains $\{a^i_v<\cdots< a^{\min\{j,r\}}\}$ and $\{b^i<\cdots<a^j_v\}$.
If $r\geq j$ then by $(3)$, there is a saturated chain $\bfc$ from $b^i$ to $a^j$ that does not intersect $C$. Now $\{b^1,\ldots,b^{i-1}\} \cup \bfc \cup \{a^{j+1},\ldots,a^r\}$ gives a maximal chain in $P$ that does not intersect $C$, which is in contradiction with our assumption that $C$ is a minimal vertex cover.
If $r<j$ then by $(4)$, there is a saturated chain $\bfc$ from $b^i$ to $a^r$. In this case, we end up with the maximal chain $\{b^1,\ldots,b^{i-1}\} \cup \bfc$, which leads us to the same contradiction.
\end{proof}

\begin{remark}
If $P$ is a pure poset then the condition $(4)$ is superfluous.
The following example shows that the conditions $(3)$ and $(4)$ in Theorem \ref{thm-unmixedness} can not be replaced by the following weaker conditions.
%Let $P$ be a ranked poset that satisfies the condition $(1)-(2)$ of Theorem \ref{thm-unmixedness}.
%Suppose $P$ also satisfies the following.
\begin{enumerate}
\item[$(3)'$] For $i<j$, if $\{a^i_u < a^{i+1} < \cdots < a^{j-1} < a^j_v\}$ and $\{a^i_v < b^{i+1} < \cdots < b^{j-1} < a^j_w\}$ are two saturated chains in $P$ (with $a^i_u, a^i_v \in P_i$ and $a^j_v,a^j_w \in P_j$), then there is saturated chain $\{a^i_u < c^{i+1} < \cdots <c^{j-1} <a^j_w\}$ in $P$.
\item[$(4')$] For $i<j<k$, if $\{a^i_u < a^{i+1} < \cdots < a^{k-1} < a^k_v\}$ and $\{a^i_v < b^{i+1} < \cdots < b^{j-1} < a^j_w\}$ are two saturated chains in $P$ (with $a^i_u, a^i_v \in P_i$, $a^k_v \in P_k$ and $a^j_w \in P_j\cap \max(P)$), then there is a saturated chain $\{a^i_u < c^{i+1} < \cdots <c^{j-1} <a^j_w\}$ in $P$.
\end{enumerate}
%Then $P$ may not be an unmixed poset.
\end{remark}

\begin{example}
Consider the following poset.
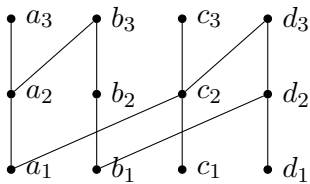
\begin{center}
 \begin{tikzpicture}[scale=0.75, vertices/.style={draw, fill=black, circle, inner sep=1pt}]
              \node [vertices, label=right:{${a}_{1}$}] (0) at (-2.25+0,0){};
              \node [vertices, label=right:{${b}_{1}$}] (2) at (-2.25+1.5,0){};
              \node [vertices, label=right:{${c}_{1}$}] (4) at (-2.25+3,0){};
              \node [vertices, label=right:{${d}_{1}$}] (6) at (-2.25+4.5,0){};
              \node [vertices, label=right:{${a}_{2}$}] (1) at (-2.25+0,1.33333){};
              \node [vertices, label=right:{${b}_{2}$}] (3) at (-2.25+1.5,1.33333){};
              \node [vertices, label=right:{${c}_{2}$}] (5) at (-2.25+3,1.33333){};
              \node [vertices, label=right:{${d}_{2}$}] (7) at (-2.25+4.5,1.33333){};
              \node [vertices, label=right:{${a}_{3}$}] (8) at (-2.25+0,2.66667){};
              \node [vertices, label=right:{${b}_{3}$}] (9) at (-2.25+1.5,2.66667){};
              \node [vertices, label=right:{${c}_{3}$}] (10) at (-2.25+3,2.66667){};
              \node [vertices, label=right:{${d}_{3}$}] (11) at (-2.25+4.5,2.66667){};
      \foreach \to/\from in {0/5, 0/1, 1/8, 1/9, 2/7, 2/3, 3/9, 4/5, 5/10, 5/11, 6/7, 7/11}
      \draw [-] (\to)--(\from);
 \end{tikzpicture}
\end{center}
This poset is pure and it satisfies the condition $(3)'$ above, but it is not unmixed. 
\end{example}

\section{Cohen-Macaulay flag ideals}
\label{sec-CM}

In \cite{Herzog-Hibi-02}, J.~Herzog and T.~Hibi showed that the edge ideal of a bipartite graph $G$ on vertex set $A\cup B$ is Cohen-Macaulay if and only if $|A|=|B|$ and there is a labeling
$A =\{a_1,\ldots,a_t\}$ and $B=\{b_1,\ldots,b_t\}$ such that
\begin{itemize}
\item[(a)]\label{a1} $\{a_i,b_i\}$ are edges of $G$, for $i=1,\ldots,t$;
\item[(b)] if $\{a_i,b_j\}$ and $\{a_j,b_k\}$ are edges of $G$, then $\{a_i,b_k\}$ is an edge;
\item[(c)] if $\{a_i,b_j\}$ is an edge then $i\leq j$.
\end{itemize} 
Now let $P$ be a poset with ground set $\{p_1,\ldots,p_t\}$.
Define a binary relation on $P$ with $p_i\leq p_j$ if and only if $\{a_i,b_j\}$ is an edge of $G$.
The condition (a), shows that $\leq$ is reflexive, (b) shows that it is transitive and finally (c) implies that $\leq$ is anti-symmetric. Therefore, they actually showed that Cohen-Macaulay edge ideals of bipartite graphs are exactly the quadratic letterplace ideals of posets.
In Theorem \ref{thm-mainCM}, we generalize their result to the class of flag ideals of graded posets.
The condition $(2)$ of Theorem \ref{thm-mainCM}, is similar to the condition (a) above, conditions $(3)$ and $(4)$ resemble the transitivity condition (b) and finally $(5)$ is similar to the condition (c).

\begin{lemma} \label{lem-localization}
Let $I$ be a monomial ideal in a polynomial ring $R=\kk[x_1,\ldots,x_n]$.
Let $G(I)=\{g_1,\ldots,g_s\}$ be the unique set of minimal generators of $I$.
For each $i$, let $g'_i$ be the monomial obtained from $g_i$ after replacing $x_1$ with $1$ (evaluating $x_1 = 1$) in $g_i$.
Let $I'$ be the monomial ideal generated by $\{g'_1,\ldots,g'_s\}$ in $\kk[x_2,\ldots,x_n]$.
If $I$ is Cohen-Macaulay then $I'$ is also Cohen-Macaulay.
\end{lemma}

\begin{proof}
Since $R/I$ is Cohen-Macaulay, its localization $(R/I)_{x_1}$ is also a Cohen-Macaulay ring.
We have
\[
(R/I)_{x_1} \cong \kk[x_2,\ldots,x_n]/I' \otimes_\kk \kk[x_1,x_1^{-1}].
\] 
The assertion now follows from \cite[Theorem 2.1]{Bouchiba-Kabbaj-01} stated below.
\end{proof}

\begin{theorem}[S.~Bouchiba, S.~Kabbaj]
Let $A$ and $B$ be two (commutative with identity) $\kk$-algebras such that $A\otimes_\kk B$ is noetherian.
Then $A\otimes_\kk B$ is Cohen-Macaulay if and only if $A$ and $B$ are Cohen-Macaulay rings.
\end{theorem}

The rank selection lemma was proved by R.~Stanley in \cite[Theorem 4.3]{Stanley-01} for the Stanley-Resiner ring of balanced Cohen-Macaulay complexes. See also \cite[Theorem 6.4]{Baclawski-01}. Here we show that a similar result holds for flag ideals of graded posets as well.
 
\begin{lemma}[Rank selection for Cohen-Macaulayness] \label{pro-rankselection}
Let $P$ be a finite poset of rank $r$ such that $\mcF(P)$ is Cohen-Macaulay.
\begin{enumerate}
\item If $P$ is graded then for any subset $S = \{i,i+1,\ldots,i+j\}$ of $[r]$, $\mcF(P,S)$ is Cohen-Macaulay;
\item If $P$ is pure then for any subset $S\subseteq [r]$, $\mcF(P,S)$ is Cohen-Macaulay.
\end{enumerate}
\end{lemma}

\begin{proof}
In both cases, let $A=\{p\in P|r_P(p)\notin S\}$. The ideal $\mcF(P,S)$ is obtained from $\mcF(P)$ by substituting all the variables in $A$ by $1$ in its generators. The assertion now follows from Lemma \ref{lem-localization}.
\end{proof}

\begin{lemma}
Let $P$ be a finite poset. If $\mcF(P)$ is Cohen-Macaulay then $P$ is unmixed.
In particular, if $P$ is a pure poset and $\mcF(P)$ is Cohen-Macaulay then the sets of elements of the same rank have the same cardinality.
\end{lemma}

\begin{proof}
If $\mcF(P)$ is Cohen-Macaulay then its Alexander dual has a linear resolution. The Alexander dual is minimally generated by $\prod_{p\in C} x_p$ where $C$ is a minimal vertex cover. This shows that all the minimal vertex covers of $P$ have the same cardinality.
For the second part, note that for any $1\leq i\leq \overline{r}(P)$, the elements of rank $i$ form a minimal vertex cover of $P$.
\end{proof}

%\begin{proposition} \label{pro-cohenmacaulay}
%Let $P$ be a ranked poset of rank $r$. If $\mcF(P)$ is Cohen-Macaulay then for $i=1,\ldots,r-1$, $P_{i,i+1}$ is a Cohen-Macaulay bipartite graph.
%\end{proposition}

%\begin{proof}
%Let $r=\overline{r}(P)$.
%For each $i$, $1\leq i\leq r-1$, the Hasse diagram of the subposet $P_{i,i+1}$ is a bipartite graph with partition $P_i\cup P_{i+1}$. Since $\mcF(P)$ is Cohen-Macaulay, $|P_i|=|P_{i+1}|$. The partial flag ideal $\mcF(P,\{i,i+1\})$ is exactly the edge ideal of this bipartite graph and by Proposition \ref{pro-rankselection}, it is a Cohen-Macaulay bipartite graph.
%The poset $P$ is clearly the tower product
%\[
%P_{1,2} \vDash P_{2,3} \vDash \cdots \vDash P_{r-1,r}
%\]
%of Cohen-Macaulay bipartite graphs.
%\end{proof}

Let $P$ be a pure poset of rank $r$.
If $\mcF(P)$ is Cohen-Macaulay then for $i=1,\ldots,r-1$, the Hasse diagram of $P_{i,i+1}$ is a Cohen-Macaulay bipartite graph but not conversely.
A complete characterization of Cohen-Macaulay bipartite graphs is given in \cite[Theorem 9.1.13]{Herzog-Hibi-01}.

\begin{example} Consider the poset $P$ in Example \ref{exp-NotUnmixed}.
The Hasse diagram of each $P_{i,i+1}$ is a Cohen-Macaulay bipartite graphs but $\mcF(P)$ is not Cohen-Macaulay.
\end{example}

\begin{definition}
Let $R = \kk[x_1,\ldots,x_n]$ be a polynomial ring.
Following \cite{Kokubo-Hibi-01}, a monomial ideal $I\subseteq R$ is called {\it weakly polymatroidal} if
\begin{enumerate}
\item $I$ is generated in a single degree;
\item if for two minimal generators $m_1 = x_1^{a_1}\cdots x_n^{a_n}$ and $m_2=x_1^{b_1}\cdots x_n^{a_n}$ of $I$, there exists an integer $t$ with
\[
a_1=b_1, \ldots, a_{t-1}=b_{t-1}, \text{ and } a_t>b_t,
\]
then there exists some $s>t$ such that $x_t(m_2/x_s)$ is also a minimal generator of $I$.
\end{enumerate} 
\end{definition}

\setcounter{claim}{0}
\begin{theorem}
\label{thm-mainCM}
Let $P$ be a graded poset of rank $r$. The ideal $\mcF(P)$ is Cohen-Macaulay if and only if $P$ satisfies the following conditions.
%\begin{enumerate}
%\item $|P_1| \geq |P_2| \geq \cdots \geq |P_r|$;
%\item for each $1\leq i\leq r$ there is an ordering $P_i = \{a^i_1,\ldots,a^i_{t_i}\}$ such that
%for each $2\leq i\leq r$, and each $1\leq u\leq t_i$, $a^{i-1}_u \leq a^i_u$ in $P$;
%%for each $u\in [t]$, $a^1_u<\cdots< a^r_u$ forms a maximal chain in $P$;
%\item for $i<j$ if $\{a^i_u < a^{i+1} < \cdots < a^{j-1} < a^j_v\}$ and $\{a^i_v < b^{i+1} < \cdots < b^{j-1} < a^j_w\}$ are two saturated chains in $P$ (where $a^i_u, a^i_v \in P_i$ and $a^j_v,a^j_w \in P_j$), then there is a saturated chain $\{a^i_u < c^{i+1} < \cdots <c^{j-1} <a^j_w\}$ such that 
%\[
%\{c^{i+1},\ldots,c^{j-1}\} \subseteq \{a^{i+1} , \cdots , a^{j-1} \} \cup \{b^{i+1} , \cdots , b^{j-1}\}.
%\]
%That is all the intermediate elements in the chain $\{a^i_u < c^{i+1} < \cdots < c^{j-1} <a^j_w\}$ are among the elements of the two given chains.
%\item for $i<j<k$, if $\{a^i_u < a^{i+1} < \cdots < a^{k-1} < a^k_v\}$ and $\{a^i_v < b^{i+1} < \cdots < b^{j-1} < a^j_w\}$ are two saturated chains in $P$ (where $a^i_u, a^i_v \in P_i$, $a^k_v \in P_k$ and $a^j_w \in P_j\cap \max(P)$), then there is a saturated chain $\{a^i_u < c^{i+1} < \cdots <c^{j-1} <a^j_w\}$ such that 
%\[
%\{c^{i+1},\ldots,c^{j-1}\} \subseteq \{a^{i+1} , \cdots , a^{j-1} \} \cup \{b^{i+1} , \cdots , b^{j-1}\}.
%\]
%\item for any $1\leq i\leq r-1$, if $p^i_u < p^{i+1}_v$ then $u \leq v$.
%\end{enumerate}

\begin{enumerate}
\item $|P_1| = |\max(P)|$;
\item Let $|P_1|=t_1$, $P_1 = \{a_1^1,\ldots,a_{t_1}^1\}$ and $\max(P) = \{b_1,\ldots,b_{t_1}\}$.
The poset $P$ contains disjoint maximal chains
\[
\bfa_i = \{a_i^1,\ldots,a_i^{r_i} = b_i\}
\]
such that $\cup_{i=1}^{t_1} \bfa_i$ equals the ground set of $P$.
\item For $i<j$ if $\{a^i_u < a^{i+1} < \cdots < a^{j-1} < a^j_v\}$ and $\{a^i_v < b^{i+1} < \cdots < b^{j-1} < a^j_w\}$ are two saturated chains in $P$ (where $a^i_u, a^i_v \in P_i$ and $a^j_v,a^j_w \in P_j$), then there is a saturated chain $\{a^i_u < c^{i+1} < \cdots <c^{j-1} <a^j_w\}$ such that 
\[
\{c^{i+1},\ldots,c^{j-1}\} \subseteq \{a^{i+1} , \cdots , a^{j-1} \} \cup \{b^{i+1} , \cdots , b^{j-1}\}.
\]
%That is all the intermediate elements in the chain $\{a^i_u < c^{i+1} < \cdots < c^{j-1} <a^j_w\}$ are among the elements of the two given chains.
\item For $i<j<k$, if $\{a^i_u < a^{i+1} < \cdots < a^{k-1} < a^k_v\}$ and $\{a^i_v < b^{i+1} < \cdots < b^{j-1} < a^j_w\}$ are two saturated chains in $P$ (where $a^i_u, a^i_v \in P_i$, $a^k_v \in P_k$ and $a^j_w \in P_j\cap \max(P)$), then there is a saturated chain $\{a^i_u < c^{i+1} < \cdots <c^{j-1} <a^j_w\}$ such that 
\[
\{c^{i+1},\ldots,c^{j-1}\} \subseteq \{a^{i+1} , \cdots , a^{j-1} \} \cup \{b^{i+1} , \cdots , b^{j-1}\}.
\]
\item For any $1\leq i\leq r-1$, if $a^i_u < a^{i+1}_v$ then $u \leq v$.

\end{enumerate}
\end{theorem}

\begin{proof}
$(\Rightarrow)$
Since $\mcF(P)$ is Cohen-Macaulay, $P$ is unmixed and the first $4$ conditions follow from Theorem \ref{thm-unmixedness}.
%Consider the poset $P_{1,r}$.

%If $rank(P) = 2$ then condition (5) follows from \cite[Theorem ??]{Herzog-Hibi-01}.
%Now suppose $rank(P) = 3$ and $P$ is not a ranked poset.
%We can assume that $P$ does not have any element of rank $1$.
%Suppose elements of $P$ are labeled as in condition (2).
%Let $B_2 = P_2 \setminus \max(P)$.
%We claim that for any $t_3 \leq i \leq t_1$, and any $1\leq j\leq t_3$ there is no edge between
%$a_i^1$ and $a_j^2$ in $P_{1,2}$.
%In other words, $\{a_{t_3},\ldots,a_{t_1}\} \cup B_2$ is an independent set in $P_{1,2}$.
%Consider the subposet $Q$ of $P$ on the vertex set $P_1 \cup B_2\cup P_3$.
%%\[
%%\{a_1^1,\ldots,a^1_{t_3}\}\cup \{a_1^2,\ldots,a^2_{t_3}\}\cup \{a^3_1,\ldots,a^3_{t_3}\}.
%%\]
%By lemma \ref{lem-localization}, $\mcF(Q)$ is a Cohen-Macaulay ideal (we localize at the variables corresponding to the elements of $P_2\setminus B_2$).
%Hence it is unmixed and both $P_1$ and $B_2$ are minimal vertex covers of $Q$.

Without loss of generality we can assume that $P$ does not have any maximal element of rank $1$.
Let $Q$ be the the subposet of $P$ on the vertex set $P_1 \cup \max(P)$.
By Lemma \ref{lem-localization}, $Q$ is a Cohen-Macaulay bipartite graph with the given partition.

%Let $Q_{1,r}$ be the induced subposet of $P_{1,r}$ on vertices $(P_1 \setminus \max(P_{1,r})) \cup P_r$. The poset $Q_{1,r}$ is a Cohen-Macaulay bipartite graph with the given partition.
%Let $B_1 = P_1 \setminus \max(P_{1,r})$.
The sets $P_1$ and $\max(P)$ are both facets of the Cohen-Macaulay simplicial complex $\Delta_{\mcF(Q)}$, where $\Delta_{\mcF(Q)}$ denotes the Stanley-Reisner complex of the ideal $\mcF(Q)$.
By \cite[Lemma 9.1.12]{Herzog-Hibi-01}, $\Delta_{\mcF(Q)}$ is connected in codimension $1$. Therefore, there is a sequence of factes $\max(P) = F_0,\ldots,F_k = P_1$ such that $|F_i\cap F_{i+1}| = t_1-1$.
We can assume $F_1 = \{b_1,\ldots,b_{t_1-1},a^1_{t_1}\}$.
A similar argument to the proof of \cite[Theorem 9.1.13]{Herzog-Hibi-01}, enables us to order the elements of $P_1$ and $\max(P)$ such that for each $i=0,\ldots,t_1$,
\[
F_i = \{b_1,\ldots,b_{t_1-i},a^1_{t_1-i+1},\ldots,a^1_{t_1}\}
\]
is a facet of $\Delta_{\mcF(Q)}$.
Observe that $F_0 = \max(P)$ and $F_t = P_1$.
%Now let
%\[
%F_i = \{p^r_1,\ldots,p^r_{t-i},p^1_{t-i+1},\ldots,p^1_t\}
%\]

We reorder the elements of the poset $P$ accordingly.
If for some $1\leq i\leq r-1$, $a^i_u \leq a^{i+1}_v$ and $u>v$, then the maximal chain
\[
\{a^1_u < \cdots < a^i_u < a^{i+1}_v < a^{i+2}_v < \cdots < a_v^{r_v}\}
\]
implies that $a^1_u \leq b_v$. This contradicts with the fact that the complement of $F_v$ is a minimal vertex cover.
%This implies that for each $j= 1,\ldots,r-1$ and each $i=0,\ldots,t$,  
%\[
%F^j_i = \{p^{j+1}_1,\ldots,p^{j+1}_i,p^j_{i+1},\ldots,p^j_t\}
%\]
%is a facet of $\Delta_{\mcF(P_{j,j+1})}$. The complement of $F^j_i$ is a minimal vertex cover of $P_{j,j+1}$, since the complement of $F_i$ in $P_{1,r}$ is a minimal vertex cover.
%The construction of $F^j_i$ shows that $\{p^j_u,p^{j+1}_v\}$ cannot be an edge of $P_{j,j+1}$ if $u>v$.
This completes the proof of the only if part

\medskip
$(\Leftarrow)$
Conversely, suppose $P$ satisfies the conditions $(1)-(5)$.
By Theorem \ref{thm-unmixedness}, $(1)-(4)$ implies that $\mcF(P)$ is unmixed.
We show that the Alexander dual of $\mcF(P)$ has a linear resolution. This implies that $\mcF(P)$ is Cohen-Macaulay.

Note that for any $1\leq i \leq r-1$, the induced subposet of $P_{i,i+1}$ on the vertex set $(P_i\setminus \max(P))\cup P_{i+1}$ is a Cohen-Macaulay bipartite graph and by \cite[Theorem 9.1.13]{Herzog-Hibi-01}, there is a poset $Q_{i,i+1}$ such that $P_{i,i+1} = L(2,Q_{i,i+1})$.
Let $|P_i|=t_i$ for $i=1,\ldots,r$.
By $(2)$, for $i=1,\ldots,r-1$, $P_i = \{a_{j_1}^i,\ldots,a^i_{j_{t_i}}\}$ where $1\leq j_1 < \cdots<j_{t_i} \leq t_1$.
We assume that the poset $Q_{i,i+1}$ has the set $Q_i = \{a_{j_1},\ldots,a_{j_{t_{i+1}}}\} \subseteq \{a_1,\ldots,a_{t_1}\}$ as its ground set.
We denote the ground set of $Q_1$, i.e. the set $\{a_1,\ldots,a_{t_1}\}$ by $Q$.

Let $C$ be a minimal vertex cover of $P$.
Let $\mcJ_1 = Q$.
For each $2\leq i\leq r$, we define the subsets $\mcJ_{i}$ of $Q$ inductively as follows.
\[
\mcJ_i = \{a_j \in \mcJ_{i-1}\cap Q_i ~|~ \exists a_k \in \mcJ_{i-1} \text{ s.t. } a_k \leq a_j \text{ in } Q_{i-1,i} \text{ and } a_k^{i-1} \notin C\}.
\]
%Note that inductively, we can show that if $a^i_j\in C$ then $a_j\notin \mcJ_{i+1}$. 
Each $\mcJ_i$ is a poset filter in $Q_{i-1,i}$.
Suppose $a_j\in \mcJ_i$ and $a_{j} \leq a_{j'}$ in $Q_{i-1,i}$.
By definition $a_j\in \mcJ_{i-1}$ and there is some $a_{k_1} \in \mcJ_{i-1}$ such that $a_{k_1}\leq a_j$ in $Q_{i-1,i}$ and $a_{k_1}^{i-1}\notin C$.
Suppose $a_j^{i-1}\in C$. Since $a_{k_1} \in \mcJ_{i-1}$, we can inductively construct a saturated chain $\{a^1_{k_{i-1}},\ldots, a^{i-1}_{k_1}\}$ that does not intersect $C$.
Now the chain $D =\{a_{k_{i-1}}^1 < \cdots < a_{k_1}^{i-1} < a_j^{i}< \cdots <a_j^r \}$ is a maximal chain in $P$ that does not intersect $C$ which is a contradiction. Therefore, $a_j^{i-1}$ does not lie in $C$ and this implies that $a_{j'}\in \mcJ_i$ by definition. Hence $\mcJ_i$ is a poset filter.

Therefore, we have a sequence
\[
\mcJ : \emptyset \subseteq \mcJ_r \subseteq \cdots \subseteq \mcJ_1 = Q
\]
such that for $i=2,\ldots,r-1$, each $\mcJ_i$ is a poset filter in $Q_{i-1,i}$.
Now define a map $\phi_\mcJ : Q  \to [r]$ that maps
\begin{align*}
\mcJ_r  & \mapsto r\\
\mcJ_{r-1} \backslash \mcJ_r & \mapsto r-1\\
\vdots\\
\mcJ_1 \setminus \mcJ_2 & \mapsto 1
\end{align*}

\begin{claim}
The assignment $(a_i,\phi_\mcJ(a_i)) \mapsto a_i^{\phi_\mcJ(a_i)}$ gives a one-to-one correspondence between the graph of $\phi_\mcJ$ and the vertex cover $C$.
\end{claim}

\begin{proof}[Proof of claim]
Let $a_i$ be an element of $Q$. If $\phi_\mcJ(a_i) = j$, where $1\leq j\leq r-1$, then $a_i\in \mcJ_{j}\backslash \mcJ_{j+1}$.
Indeed, $a_i \leq a_i$ in $Q_{j,j+1}$. If $a_i^{j} \notin C$ then by definition $a_i\in \mcJ_{j+1}$ which is against our assumption. Hence $a_i^{j} \in C$.
Now suppose $\phi_\mcJ(a_i) = r$, i.e. $a_i\in \mcJ_r$.
In this case, suppose $a_i^r\notin C$. Let $a_{i_1}=a_i$.
Since $a_{i_1}\in \mcJ_r$, there exists some $a_{i_2}\in \mcJ_{r-1}$ such that $a_{i_2}\leq a_{i_1}$ in $Q_{r-1,r}$ and $a_{i_2}^{r-1}\notin C$. Inductively, we can find elements $a_{i_j}$, $j=1,\ldots,r$ such that $a_{i_j}^{r-j+1}$ is not an element of $C$ and $\{a_{i_r}^1 < \cdots <a_{i_1}^r\}$ forms a maximal chain in $P$.
This contradicts with the fact that $C$ is a vertex cover. Therefore, we see that if for some $a_i\in Q$, $\phi_\mcJ(a_i) = r$ then $a_i^r\in C$.
The assertion now follows from the fact that $|Q| = |C|$.
\end{proof}

So far we have shown that any minimal vertex cover $C$ gives a sequence of poset filters in the posets $Q_{1,2},\ldots,Q_{r-1,r}$.
Now we show that any such sequence gives a minimal vertex cover of $P$.
Let
\[
\mcJ : \emptyset \subseteq \mcJ_r \subseteq \cdots \subseteq \mcJ_1 = Q
\]
be a filtration of $Q$ such that $\mcJ_i$ is a poset filter in $Q_{i-1,i}$.
We show that
\[C=\{ a_1^{\phi_\mcJ(a_1)},\ldots,a_{t_1}^{\phi_\mcJ(a_{t_1})}\}\]
is a minimal vertex cover of $P$.
Let $D=\{a_{i_1}^1 < \cdots < a_{i_r}^r\}$ be a maximal chain in $P$ that does not intersect $C$.
Observe that $a_{i_1}^1 \notin C$ implies that $\phi_\mcJ(a_{i_1}) \neq 1$. Hence $a_{i_1} \notin \mcJ_1 \backslash \mcJ_{2}$, i.e. $a_{i_1}\in \mcJ_{2}$. Since $a_{i_{1}} \leq a_{i_2}$ in $Q_{1,2}$ and $\mcJ_{2}$ is a poset filter in $Q_{1,2}$, we have $a_{i_{2}}\in \mcJ_{2}$.
By induction we can show that $a_{i_r}\in \mcJ_r$, and by definition of $\phi_\mcJ$, $a_{i_r}^r \in C$ which is a contradiction. This implies the desire conclusion that $C$ is a vertex cover. Observe that $C$ is a minimal vertex cover since $|C| = t_1$.

Recall that the minimal generators of $\mcF(P)^A$ are in one-to-one correspondence with minimal vertex covers of $P$.

\begin{center}
\begin{tabular}{p{4.5cm}cp{5cm}}
$C =\{ a^{i_1}_1,\ldots,a^{i_{t_1}}_{t_1}\}$ is a minimal vertex cover of $P$
& $\Leftrightarrow$ &
$x_C = x_{a^{i_1}_1} \cdots x_{a^{i_{t_1}}_{t_1}}$ is a minimal generator of $\mcF(P)$
\end{tabular}
\end{center}

By above discussion, any minimal vertex cover is in one-to-one correspondence with filtrations of $Q$,
\[
\mcJ : \emptyset \subseteq \mcJ_r \subseteq \cdots \subseteq \mcJ_1 = Q
\]
such that for $i=2,\ldots,r$, the set $\mcJ_i$ is a poset filter of $Q_{i-1,i}$.
Let $\mcJ(P)$ be the set of all such filtrations.

We assign to any element $\mcJ$ of $\mcJ(P)$, the monomial $m_\mcJ = x_C$ in the polynomial ring $S=\kk[x_P]$, where $C=\{ a_1^{\phi_\mcJ(a_1)},\ldots,a_{t_1}^{\phi_\mcJ(a_{t_1})}\}$.
By the observations above
\[
\mcF(P)^A = <m_\mcJ ~|~ \mcJ \in \mcJ(P)>.
\]

\begin{claim}
The ideal $\mcF(P)^A = <m_\mcJ ~|~ \mcJ \in \mcJ(P)>$ is weakly polymatroidal.
\end{claim}

\begin{proof}[Proof of claim]
%Let $S=\kk[x_P]$.
In the following proof we denote a variable $x_{a^i_j}$ by $a^i_j$ for simplicity.
We assume that $a^i_j > a^{i'}_{j'}$ if
\begin{enumerate}
\item[i.] $i > i'$, or
\item[ii.] $i=i'$ and $j > j'$.
\end{enumerate}
This gives us a variable order on the polynomial ring $S$ and we show that $\mcF(P)^A$ is weakly polymatroidal with respect to this variable order.

%We show that $\mcF(P)^A$ is weakly polymatroidal with respect to the variable order
%\[
%a_1^1 > a_2^1 > \cdots > a_t^1 > a_1^2 > \cdots > a_t^r.
%\]
Let $\mcJ$ and $\mcI$ be two filtrations in $\mcJ(P)$.
Suppose for all $a_{j'}^{i'} > a_{j}^i$, $\deg_{a_{j'}^{i'}} m_{\mcI} = \deg_{a_{j'}^{i'}} m_{\mcJ}$ and $\deg_{a_{j}^{i}} m_{\mcI} > \deg_{a_{j}^{i}} m_{\mcJ}$.

We show that
\begin{enumerate}
\item $\mcI_{i'} = \mcJ_{i'}$ for all $i'>i$;
\item $a_{j'} \in \mcI_{i} \Leftrightarrow a_{j'} \in \mcJ_i$ for all $j'>j$.
\end{enumerate}
Let $i'=r$. Note that $a_k \in \mcI_r$ (resp. $a_k\in \mcJ_r$) if and only if $\deg_{a_k^r} m_\mcI = 1$
(resp. $\deg_{a_k^r} m_\mcJ = 1$). In this case the assertion in $(1)$ follows from the equality $\deg_{a_k^r} m_\mcI = \deg_{a_k^r} m_\mcJ$.
Now suppose $i' > i$ and for all $i''>i'$ we have $\mcI_{i''} = \mcJ_{i''}$.
We have $a_k \in \mcI_{i'}\backslash \mcI_{i'+1}$ (resp. $a_k\in \mcJ_{i'}\backslash\mcJ_{i'+1}$) if and only if $\deg_{a_k^{i'}} m_\mcI = 1$
(resp. $\deg_{a_k^{i'}} m_\mcJ = 1$) and the assertion follows similarly.

Now suppose $a_{j'}\in \mcI_{i}$. If $a_{j'}\in \mcI_{i+1}$ then by $(1)$, $a_{j'}\in \mcJ_{i+1} \subseteq \mcJ_{i}$.
Otherwise if $a_{j'}\in \mcI_{i}\backslash \mcI_{i+1}$ then $\deg_{a_{j'}^i} m_\mcI = 1$. Since  $\deg_{a_{j'}^i} m_\mcI = \deg_{a_{j'}^i} m_\mcJ$, we have $a_{j'}\in \mcJ_{i}\backslash \mcJ_{i+1}$.
Conversely, we can show that $a_{j'}\in \mcJ_{i}$ implies $a_{j'}\in \mcI_i$. This completes the proof of statements $(1)$ and $(2)$ above.

By our assumption $\deg_{a_{j}^{i}} m_{\mcI} > \deg_{a_{j}^{i}} m_{\mcJ}$. This implies that $a_j \in \mcI_i \backslash \mcI_{i+1}$ and $a_j\notin \mcJ_i \backslash \mcJ_{i+1}$. Since $\mcI_{i+1} = \mcJ_{i+1}$, there exists some $i'<i$ such that $a_j\in \mcJ_{i'}\backslash \mcJ_{i'+1}$.
Consider the filtration
\[
\mcK : \emptyset \subseteq \mcJ_r \subseteq \cdots \subseteq \mcJ_{i+1}
\subseteq \mcJ_{i} \cup \{a_j\} \subseteq \cdots \subseteq \mcJ_{i'+1} \cup \{a_j\} \subseteq
\mcJ_{i'} \subseteq \cdots \subseteq \mcJ_1
\]
of $Q$.
We claim that $\mcK \in \mcJ(P)$.
We only need to show that for any $i' < k< i$, $\mcJ_k \cup \{a_j\}$ is a poset filter of $Q_{k-1,k}$.
Since $\mcI_k$ is a poset filter containing $a_j$, for any $a_{j'} > a_j$ in $Q_{k-1,k}$ we have $a_{j'}\in \mcI_k$. By $(1)$ and $(2)$, we see that $a_{j'} \in \mcJ_k$. Hence $\mcK\in \mcJ(P)$.

Now by construction we have $m_\mcK = a^i_j (m_\mcJ / a^{i'}_j)$ and this completes the proof that $\mcF(P)^A$ is weakly polymatroidal.
\end{proof}
By \cite[Theorem 1.4]{Kokubo-Hibi-01}, polymatroidal ideals have linear quotients. In particular, they have linear resolutions. 
\end{proof}

The above theorem introduces a huge class of Cohen-Macaulay ideals not necessarily generated in a single degree.

\begin{example}
Let $P$ be the following graded poset.
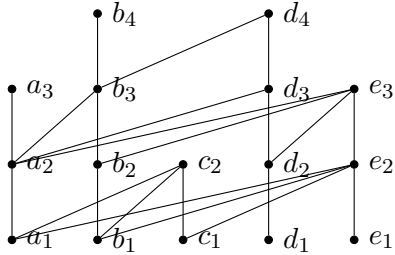
\begin{center}
\begin{tikzpicture}[scale=.75, vertices/.style={draw, fill=black, circle, inner sep=1pt}]
              \node [vertices, label=right:{${a}_{1}$}] (0) at (-3+0,0){};
              \node [vertices, label=right:{${b}_{1}$}] (4) at (-3+1.5,0){};
              \node [vertices, label=right:{${c}_{1}$}] (6) at (-3+3,0){};
              \node [vertices, label=right:{${d}_{1}$}] (7) at (-3+4.5,0){};
              \node [vertices, label=right:{${e}_{1}$}] (9) at (-3+6,0){};
              \node [vertices, label=right:{${a}_{2}$}] (1) at (-3+0,1.33333){};
              \node [vertices, label=right:{${b}_{2}$}] (5) at (-3+1.5,1.33333){};
              \node [vertices, label=right:{${c}_{2}$}] (2) at (-3+3,1.33333){};
              \node [vertices, label=right:{${d}_{2}$}] (8) at (-3+4.5,1.33333){};
              \node [vertices, label=right:{${e}_{2}$}] (3) at (-3+6,1.33333){};
              \node [vertices, label=right:{${a}_{3}$}] (10) at (-3,2.66667){};
              \node [vertices, label=right:{${b}_{3}$}] (11) at (-3+1.5,2.66667){};
              \node [vertices, label=right:{${d}_{3}$}] (12) at (-3+4.5,2.66667){};
              \node [vertices, label=right:{${e}_{3}$}] (13) at (-3+6,2.66667){};
              \node [vertices, label=right:{${b}_{4}$}] (14) at (-3+1.5,4){};
              \node [vertices, label=right:{${d}_{4}$}] (15) at (-3+4.5,4){};
      \foreach \to/\from in {0/1, 0/2, 0/3, 1/10, 1/11, 1/12, 1/13, 3/13, 4/5, 4/2, 4/3, 5/13, 5/11, 6/2, 6/3, 7/8, 8/12, 8/13, 9/3, 11/14, 11/15, 12/15}
      \draw [-] (\to)--(\from);
      \end{tikzpicture}
\end{center}
The conditions $(1)$ and $(4)$ are obviously satisfied. One can also show that the conditions $(2)$ and $(3)$ hold as well.
The ideal $\mcF(P)$ has $17$ generators and it can be checked by a computer algebra system like Macaulay2, \cite{Grayson-Stillman-01}, that it is a Cohen-Macaulay ideal.

\end{example}

\section{Multigraded Betti numbers of flag ideals}
\label{sec-multidegrees}

The multigraded Betti numbers of a squarefree monomial ideal $I$ can be computed from the (co)homology of restrictions of the Stanley-Reisner complex $\Delta_I$ via the Hochster formula.
Let $P$ be a finite graded poset and let $\mcF(P)$ be its flag ideal.
In this section we show that the restriction of the simplicial complex $\Delta_{\mcF(P)}$ to a multidegree is the join of simplicial complexes arising from bipartite graphs on elements of consecutive rank.
The results of this section are inspired by the work A. D'Al{\`i}, G. Fl{\o}ystad and the athor in \cite{DAli-Floystad-Nematbakhsh-01}.

The following simple lemma is proved as part of \cite[Proposition 3.1]{JuhnkeKubitzke-Katthan-SaeediMadani-01} for monomial ideals using lcm-lattices, but it also holds for polynomial ideals.
This lemma shows that computation of multigraded Betti numbers of a flag ideal can be reduced to its connected components.

\begin{lemma}
Let $R = \kk[x_1,\ldots , x_n, y_1, \ldots , y_m]$.
Let $I$ be an ideal in variables $x_1, \ldots , x_n$
and $J$ be an ideal in variables $y_1,\ldots , y_m$. Then the tensor product of minimal free resolutions of $R/I$ and $R/J$ is a minimal free resolution of $R/(I + J)$. In particular, for all $i \geq 0$, $A \subseteq \{x_1,\ldots,x_n\}$ 
and $B\subseteq \{y_1,\ldots,y_m\}$ we have
\[
\beta_{i, A \cup B}(R/I + J) = \sum_{j+k=i} \beta_{j, A}(R/I) \beta_{k, B}(R/J).
\]
\end{lemma}

\begin{proof}
By K\"unneth spectral sequence it is enough to show that for all $i \geq 1$,
$\Tor^R_i (R/I, R/J) = 0$.
Let
\[
0 \to F_t \stackrel{\partial_t}{\longrightarrow} \cdots \stackrel{\partial_2}{\longrightarrow}  F_1 \stackrel{\partial_1}{\longrightarrow} F_0 \to 0
\]
be a minimal free resolution of $R/I$. Since any differential matrix $\partial_i$ consists of variables
in $x_1, \ldots , x_n$, the functor $- \otimes R/J$ preserves exactness of this sequence except in degree
$0$. This shows that all the torsion modules vanish.
\end{proof}

%Let $P$ be a ranked poset of rank $r$ and let $r_P$ be a rank function.
%For any $i$, $i=1,\ldots,r-1$ let $G(i)$ be a bipartite graph on partition
%\[r_P^{-1}(i) \backslash \max(P) \cup r_P^{-1}(i+1)\]
%and edges $\{p,q\}$ such $p$ is an element of rank $i$, $q$ is an element of rank $i+1$ and $p\leq q$ in $P$.

\medskip
Let $X_{AB}$ be the simplicial complex on vertex set $A\cup B$ with Stanley-Reisner ideal $I_{AB}$.
Similarly, let $X_{B_0C}$ be a simplicial complex on vertex set $B_0\cup C$ where $B_0\subseteq B$.
We denote the Stanley-Resiner ideal of $X_{B_0C}$ by $I_{B_0C}$.
We also assume that no minimal generator of $I_{AB}$ and $I_{B_0C}$ is divisible by a monomial of form $x_{b_1}x_{b_2}$ where $b_1$ and $b_2$ are elements of $B$.
Let $X$ be a a simplicial complex on $A\cup B\cup C$ whose Stanley-Resiner ideal is generated by monomials
\begin{enumerate}
\item $x_\bfa \in I_{AB}$ with $\bfa \subseteq A$;
\item $x_\bfc \in I_{B_0C}$ with $\bfc \subseteq C$;
\item $x_\bfa x_b x_\bfc$ where $\bfa\subseteq A$, $b\in B_0$ and $\bfc\subseteq C$ such that $x_\bfa x_b\in I_{AB}$ and $x_bx_\bfc \in I_{B_0C}$;
\item $x_\bfa x_b \in I_{AB}$ where $\bfa \subseteq A$ and $b\in B\backslash B_0$.
\end{enumerate}
The simplicial complex $X$ consists of all subsets $\bfa \cup \bfb \cup \bfc$ (with $\bfa \subseteq A, \bfb \subseteq B$ and $\bfc\subseteq C$) such that
\begin{enumerate}
\item for all $b \in \bfb \cap B_0$, either i) $\bfa \cup \{b\} \in X_{AB}$ and $\bfc \in X_{B_0C}$,  or ii) $\{b\}\cup \bfc \in X_{B_0C}$ and $\bfa\in X_{AB}$;
\item for all $b\in \bfb \cap (B\backslash B_0)$, $\bfa \cup \{b\} \in X_{AB}$ and $\bfc \in X_{B_0C}$;
\end{enumerate}
or equivalently,
\begin{enumerate}
\item $\bfa \in X_{AB}$ and $\bfc\in X_{B_0C}$;
\item for all $b \in \bfb\cap B_0$, either $\bfa \cup \{b\} \in X_{AB}$, or $\{b\} \cup \bfc \in X_{B_0C}$;
\item for all $b\in \bfb \cap (B\backslash B_0)$, $\bfa \cup \{b\} \in X_{AB}$.
\end{enumerate}

\begin{theorem} \label{thm-towerjoin}
Let $X_{AB}, X_{B_0C}$ and $X$ be as above.
The join $X_{AB} \ast X_{B_0C}$ is homotopy equivalent to $X$.
\end{theorem}

\begin{proof}
Let $B_0 = \{b_1,\ldots,b_p\}$ and let $B_0''=\{b''_1,\ldots,b''_p\}$ be a copy of $B_0$.
Similarly, let $B = \{b_1,\ldots,b_p,b_{p+1},\ldots,b_s\}$ and let $B' = \{b'_1,\ldots,b'_p,b_{p+1},\ldots,b_s\}$.
We get simplicial complexes on $X_{AB'}$ and $X_{B_0''C}$ which are copies of $X_{AB}$ and $X_{B_0C}$ respectively.
For $i=1,\ldots,p$, let $B_{\leq i}$ be the subset $\{b_1,\ldots,b_i\}$ of $B$.
For $i=0,\ldots,p-1$, we denote the subset $\{b'_{i+1},\ldots,b'_p,b_{p+1},\ldots,b_s\} \subseteq B'$ by $B'_{>i}$. Similarly define $B''_{>i}$ to be the subset $\{b''_{i+1},\ldots,b''_{p}\}$ of $B''$.
We set $B_{\leq 0}$ and $B''_{>p}$ to be the empty set.

For $i=0,\ldots,p$, let $X_i$ be the simplicial complex on the vertex set
\[
A \cup B_{\leq i} \cup B'_{>i} \cup B''_{>i} \cup C
\]
consisting of all subsets
\[
\bfa \cup \bfb_{\leq i} \cup \bfb'_{>i} \cup \bfb''_{>i} \cup \bfc
\]
satisfying
\begin{enumerate}
\item $\bfa \cup \bfb'_{>i}$ is in $X_{AB'}$;
\item $\bfb''_{>i} \cup \bfc$ is in $X_{B''C}$;
\item $\bfa \cup \bfb_{\leq i} \cup \bfc$ is in $X$.
\end{enumerate}
It is not hard to show that each $X_i$ is a simplicial complex. Furthermore, $X_0$ is the join $X_{AB'} \ast X_{B''_0C}$ and $X_p = X$.
By using \cite[Proposition 1]{DAli-Floystad-Nematbakhsh-01}, we show that $X_i$ and $X_{i+1}$ are homotopy equivalent. The restriction of $X_{i}$ and $X_{i+1}$ to
\[
V_0 = (A \cup B_{\leq i} \cup B'_{>i} \cup B''_{>i} \cup C)\backslash \{b'_{i+1},b''_{i+1}\}
\]
is the same.
Let $R= \bfa \cup \bfb_{\leq i} \cup \bfb'_{>i+1} \cup \bfb''_{>i+1} \cup \bfc$ be a subset of $V_0$.
%Note: Note that we do not assume that $R$ is a face of $X_i$ or $X_{i+1}$, it is just a subset.
We must show that:
\begin{enumerate}
\item If $R \cup \{b'_{i+1}\}$ and $R \cup \{b''_{i+1}\}$ are in $X_i$, then $R\cup \{b'_{i+1},b''_{i+1}\}$ is in $X_i$.

\noindent
If $R \cup \{b'_{i+1}\}$ is in $X_i$ then $\bfa \cup \bfb'_{>i+1} \cup \{b'_{i+1}\}$ is in $X_{AB'}$. Similarly, $R \cup \{b''_{i+1}\} \in X_i$ implies that $\bfb''_{>i+1} \cup \{b''_{i+1}\} \cup \bfc$ is in $X_{B_0''C}$.
It follows from the criteria above that $R\cup \{b'_{i+1},b''_{i+1}\}$ is also in $X_i$.
\item If $R\cup \{b'_{i+1}\}$ or $R\cup \{b''_{i+1}\}$ is in $X_i$, then $R\cup \{b_{i+1}\}$ is in $X_{i+1}$.

\noindent
Since $R\cup \{b'_{i+1}\}$ or $R\cup \{b''_{i+1}\}$ is in $X_i$, $\bfa \cup \bfb'_{>i+1}\in X_{AB'}$ and $\bfb''_{>i+1} \cup \bfc \in X_{B''_0C}$.
It follows that if $R\cup \{b_{i+1}\}$ is not in $X_{i+1}$ then $\bfa \cup \bfb_{\leq i} \cup \{b_{i+1}\} \cup \bfc$ is not in $X$.
Note that since $R\in X_i$, $\bfa \in X_{AB}$ and $\bfc\in X_{B_0C}$.
This implies that $\bfa \cup \{b_{i+1}\}\notin X_{AB}$ and $\{b_{i+1}\} \cup \bfc \notin X_{B_0C}$.
This contradicts our assumption that $R\cup \{b'_{i+1}\}$ or $R\cup \{b''_{i+1}\}$ is in $X_i$.
\item If $R\cup \{b_{i+1}\}$ is in $X_{i+1}$, then $R\cup \{b'_{i+1}\}$ or $R\cup \{b''_{i+1}\}$ is in $X_i$.

\noindent
If $R\cup \{b_{i+1}\}$ is in $X_{i+1}$ then $\bfa \cup \bfb_{\leq i} \cup \{b_{i+1}\} \cup \bfc$ is in $X$.
Therefore, either $\bfa \cup \{b_{i+1}\} \in X_{AB}$ or $\{b_{i+1}\} \cup \bfc \in X_{B_0C}$.
In the first case we show that 
\[
R\cup \{b'_{i+1}\} = \bfa \cup \bfb_{\leq i} \cup \bfb'_{>i+1} \cup \{b'_{i+1}\} \cup \bfb''_{>i+1} \cup \bfc
\]
is in $X_i$.
\begin{enumerate}
\item Since $\bfa \cup \bfb'_{>i+1} \in X_{AB'}$ and $\bfa \cup \{b_{i+1}\}\in X_{AB}$, we have $\bfa \cup \bfb'_{>i+1} \cup \{b'_{i+1}\}$ is in $X_{AB'}$ (This is where we are using the fact that no generator of $I$ is divisible by a quadratic monomial $x_{b_1}x_{b_2}$ with $b_1,b_2\in B$);
\item Clearly $\bfb''_{>i+1} \cup \bfc$ is in $X_{B''_0C}$ and;
\item $\bfa \cup \bfb_{\leq i} \cup \bfc$ is in $X$.
\end{enumerate}
In the other case, we see that $R\cup \{b''_{i+1}\}$ is in $X_i$.
\end{enumerate}
\end{proof}

Let $G$ be a graph on the vertex set $A_1 \cup \cdots \cup A_n$ and let $B_1,\ldots,B_{n}$ be subsets of $A_1,\ldots,A_{n}$ respectively. We always assume that $B_1 = A_1$ and $B_n=\emptyset$.
%Suppose that the edge set of $G$ satisfies
%\[
%E(G) \subseteq \bigcup_{i=1}^{n-1} B_{i} \times A_{i+1}.
%\]
Suppose that for any edge $\{p,q\} \in E(G)$, there is some $1\leq i\leq n-1$ such that $p\in B_i$ and $q\in A_{i+1}$.
For $i=1,\ldots,n-1$, let $X_i$ be the simplicial complex on $B_{i}\cup A_{i+1}$ whose Stanley-Reisner ideal is the edge ideal of the graph with edges between $B_{i}$ and $A_{i+1}$ in $G$, i.e. the induced subgraph of $G$ on $B_{i}\cup A_{i+1}$.
Now let $X$ be the simplicial complex on $A_1\cup\cdots\cup A_n$ whose Stanley-Reisner ideal is generated by all monomials $a_1 a_2 \cdots a_s$ such that $s\leq n$, $a_s\in A_s\backslash B_s$, for all $1\leq i\leq s$, $a_i\in A_i$,  and for all $1 \leq j\leq s-1$, $\{a_{j},a_{j+1}\}$ is an edge of the graph $G$.

\begin{lemma} \label{lem-towerjoin}
Let simplicial complexes $X$ and $X_i$ be as above.
The simplicial complex $X$ is homotopy equivalent to the join $X_1 \ast \cdots \ast X_{n-1}$.
\end{lemma}

\begin{proof}
Let $X_{\leq i}$ be the simplicial complex on $A_1\cup\cdots\cup A_{i+1}$ whose Stanley-Resiner ideal is generated by monomials corresponding to edge paths in $G$ from $A_1$ to $A_{i+1}$.
Note that $X_{\leq 1} = X_1$ and $X_{\leq n-1} = X$.
The assertion follows from Theorem \ref{thm-towerjoin} by induction.
\end{proof}

\begin{construction} \label{con-Y}
Let $G$ be a bipartite graph on vertex set $A\cup B$. Let $X$ be the simplicial complex on $A\cup B$ whose Stanley-Resiner ideal is the edge ideal of $G$.
Let $Y$ be the simplicial complex on $B$ with faces $\bfb\subseteq B$ such that for some $a\in A$, $\bfb\cup \{a\}$ is in $X$.
Note that the vertex set of $Y$ might be empty.
It is shown in \cite[Proposition 2]{DAli-Floystad-Nematbakhsh-01} that $X$ is homotopy equivalent to suspension of $Y$.
\end{construction}

Let $P$ be a graded poset of rank $r$ with the rank function $r_P$.
For a subset $A\subseteq P$ we can write $A = \cup_{i=1}^r A_i$ where $A_i = A\cap P_i$.
Let $B_i = A_i\backslash \max(P)$ for $i=2,\ldots, r-1$ and let $B_1=A_1$ and $B_r = A_r$.
There is a graph $G$ on vertex set $A$ such that the edges are between $B_{i}$ and $A_{i+1}$ for $i=1,\ldots,r-1$ and $\{a_{i},a_{i+1}\} \in E(G)$ if and only if $a_{i+1}$ covers $a_{i}$ in $P$.
Let $X_i(A)$ be the simplicial complex whose Stanley-Resiner ideal is the edge ideal of the induced bipartite subgraph of $G$ on $B_{i}\cup A_{i+1}$.
We denote this bipartite graph by $G_i(A)$.
We also denote the complex defined on $B_i$ as in Construction \ref{con-Y} by $Y_i(A)$.

\begin{proposition}
Let $P$ be a finite graded poset of rank $r$.
Let $\Delta$ be a simplicial complex whose Stanley-Resiner ideal is the flag ideal $\mcF(P) \subseteq \kk[x_P]$  and let $A\subseteq P$ be a multidegree.
The restriction $\Delta_A$ is homotopy equivalent to the join
$X_1(A) \ast \cdots \ast X_{r-1}(A)$.
\end{proposition}

\begin{proof}
The Stanley-Resiner ideal of the simplicial complex $\Delta_A$ is generated by monomials
$x_{a_1}x_{a_2}\cdots x_{a_s}$ where
% $a_s\in (A_s\backslash B_s)$,
$a_s$ is a maximal element of rank $s$, and for $1\leq i\leq s-1$, $a_i\in A_i$  and $\{a_i,a_{i+1}\}$ is an edge of the graph $G$ defined above.
Now the assertion follows from Lemma \ref{lem-towerjoin}.
\end{proof}

We define the {\it polynomial of reduced cohomology} of a simplicial complex $X$ to be
\[
\tilde{H}(X,t) = \sum_{i\geq -1} t^i \dim_\kk \tilde{H}^i(X,\kk).
\]

Let $A$ be a multidegree and let $|A|=n$.
We call the polynomial
\[
\beta(A,t) = t^n \beta_{0,A}(\mcF(P))+t^{n-1} \beta_{1,A}(\mcF(P)) +\cdots+ t \beta_{n-1,A}(\mcF(P))
\]
the {\it Betti polynomial} of the ideal $\mcF(P)$ with respect to a multidegree $A$.

\medskip
The following is Theorem 4 in \cite{DAli-Floystad-Nematbakhsh-01}.

\begin{theorem}
Let $P$ be a finite graded poset and let $r=\overline{r}(P)$
The Betti polynomial $\beta(A,t)$ of $\mcF(P)$ is
\[
\beta(A,t) = t^r \prod_{i=1}^{r-1} \tilde{H}(X_i(A),t).
\]
\end{theorem}

Let $P$ be a graded poset with the rank function $r_P$. Let $r=\overline{r}(P)$ and let $s = \underline{r}(P) = \min\{r_P(p) | p\in \max(P)\}$. The first linear strand of $\mcF(P)$ is the $s$-linear strand.

\begin{lemma}
Let $A\subseteq P$ be a multidegree.
If for some $i$, $X_i(A)$ is the irrelevant complex then either the multidegree $A$ does not appear in the resolution of $\mcF(P)$ or $A\subseteq A_1 \cup \cdots \cup A_{i-1} \cup(A_i \backslash B_i)$ and for all $j\geq i$, $X_j(A)$ is the irrelevant complex as well.
\end{lemma}

\begin{proof}
Suppose the multidegree $A$ appears with a nonzero multigraded Betti number in the resolution of $\mcF(P)$.
We first show that if for some $i$, $X_i(A)$ is the irrelevant complex then for all $j\geq i$, $X_j(A)$ is also the irrelevant complex.
Suppose $X_i = \{\emptyset\}$. This can only occur if $B_i \cup A_{i+1} = \emptyset$. Suppose for some $j>i+1$, $A_j \neq \emptyset$ and let $j$ be the minimal integer with this property. 
Since there are no edges between $B_{j-1}$ and $A_j$, $X_{j-1}$ is contractible and the multidegree $A$ can not occur which is a contradiction.
Therefore, for all $j \geq i+1$, $A_j = \emptyset$. Thus $X_{j}=\{\emptyset\}$ for all $j\geq i$.
\end{proof}

\begin{proposition} \label{pro-firststrand}
Let $P$ be a finite graded poset with the rank function $r_P$ and let $s=\underline{r}(P)$.
The multidegree $A$ appears in the first linear strand if and only if
\begin{enumerate}
\item $A \subseteq P_{1,\ldots,s}$ and $A_s \subseteq \max(P)$;
\item for all $1\leq i\leq s$, $A_i$ is nonempty;
\item The induced subposet of $P$ on $B_i\cup A_{i+1}$ is a complete bipartite graph for all $1\leq i\leq s-1$.
\end{enumerate}
\end{proposition}

\begin{proof}
The multidegree $A$ appears in the first linear strand of $P$ if and only if $|A| = j+s$ and $\beta_{j,A}$ is nonzero. The Betti number $\beta_{j,A}$ is the coefficient of $t^s$ in $\beta(A,t)$.
Since $\beta(A,t) = t^r \prod_{i=1}^{r-1} \tilde{H}(X_i(A),t)$ at least $r-s$ simplicial complexes $X_i$ are the irrelevant complex $\{\emptyset\}$.
Note that for $1\leq i\leq s-1$, $B_i = A_i$.
Suppose $i$ is the smallest integer such that $X_i = \{\emptyset\}$.
If $i\leq s-1$ then $A \subseteq A_1\cup \cdots \cup (A_i \backslash B_i)$. But $B_i=A_i$ and
$A \subseteq A_1\cup \cdots \cup  A_{i-1}$.
Since $i$ is the smallest integer with $X_i$ being the irrelevant complex we find that $B_{i-1}$ is nonempty. Hence $X_{i-1}$ is contractible which is a contradiction.
This shows that exactly for $s\leq i \leq r-1$, $X_i$ is the irrelevant complex and
$A \subseteq A_1 \cup \cdots\cup (A_s\backslash B_s)$.
Now $\beta_{i,A}$ is the coefficient of $t^s \prod_{i=1}^{s-1} \tilde{H}(X_i(A),t)$ and each $\tilde{H}(X_i(A),t)$ must have a nonzero constant term.
By lemma \cite[Proposition 2]{DAli-Floystad-Nematbakhsh-01}, $\tilde{H}^0(X_i(A),\kk) = \tilde{H}^{-1}(Y_i(A),\kk)$ and it is nonzero if and only if $Y_i(A)$ is the irrelevant complex. This shows that for any $p\in B_i$ and any $q \in A_{i+1}$, $p \leq q$ or equivalently the induced subposet of $P$ on $B_i \cup A_{i+1}$ is a complete bipartite graph.
\end{proof}

\section{Flag ideals with linear resolutions}
\label{sec-linearRes}

Let $P$ be a graded poset with the rank function $r_P$. Let $r=\overline{r}(P)$ be the rank of $P$.
If $\mcF(P)$ has a linear resolution then all of the minimal generators should have the same degree $r$. Thus $P$ must be a pure poset.
We first consider the rank 2 poset, i.e. the posets that their Hasse diagram is a bipartite graph.
In \cite{Nagel-Reiner-01}, U. Nagel and V. Reiner show that the edge ideal of a bipartite graph has a linear resolution if and only if it is a Ferrers graph.
Here we provide another equivalent condition.

\begin{definition}
A {\it Ferrers graph} is a bipartite graph $G$ on two distinct vertex sets $\{x_1,\ldots,x_n\}$ and $\{y_1,\ldots,y_m\}$ satsidfying the following two conditions.
\begin{enumerate}
\item If $\{x_p,y_q\}$ is an edge of $G$, then for all $i<p$ and $j<q$, $\{x_i,y_j\}$ is an edge of $G$ as well;
\item $\{x_1,y_m\}$ and $\{x_n,y_1\}$ are edges of $G$. 
\end{enumerate}
\end{definition}

\begin{theorem} \label{thm-Ferrerschar}
Let $P$ be a bipartite graph.
The following are equivalent.
\begin{enumerate}
\item The flag ideal $\mcF(P)$ has a linear resolution;
\item $P$ does not have an induced subgraph of the form below;
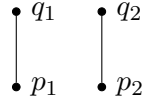
\begin{center}
\begin{tikzpicture}[scale=.75, vertices/.style={draw, fill=black, circle, inner sep=1pt}]
             \node [vertices, label=right:{${p}_{1}$}] (0) at (-.75+0,0){};
             \node [vertices, label=right:{${p}_{2}$}] (2) at (-.75+1.5,0){};
             \node [vertices, label=right:{${q}_{1}$}] (1) at (-.75+0,1.33333){};
             \node [vertices, label=right:{${q}_{2}$}] (3) at (-.75+1.5,1.33333){};
     \foreach \to/\from in {0/1, 2/3}
     \draw [-] (\to)--(\from);
     \end{tikzpicture}
\end{center}
\item $P$ is a Ferrers graph.
\end{enumerate}
\end{theorem}

\begin{proof}
$(1) \Rightarrow (2)$ Suppose $P$ has an induced subposet with cover relations $p_1\leq q_1$ and $p_2\leq q_2$. For the multidegree $A=\{p_1,p_2,q_1,q_2\}$ the simplicial complex $X_1(A)$ is
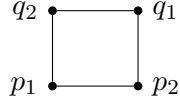
\begin{center}
\begin{tikzpicture}[scale=.75, vertices/.style={draw, fill=black, circle, inner sep=1pt}]
             \node [vertices, label=left:{${p}_{1}$}] (0) at (-.75+0,0){};
             \node [vertices, label=right:{${p}_{2}$}] (2) at (-.75+1.5,0){};
             \node [vertices, label=left:{${q}_{2}$}] (1) at (-.75+0,1.33333){};
             \node [vertices, label=right:{${q}_{1}$}] (3) at (-.75+1.5,1.33333){};
     \foreach \to/\from in {0/2, 1/3, 0/1, 2/3}
     \draw [-] (\to)--(\from);
     \end{tikzpicture}
\end{center}
which has nonvanishing cohomology. Thus the multidegree $A$ appears in the minimal resolution of $\mcF(P)$ but it is not on the first linear strand by Proposition \ref{pro-firststrand}.

\medskip
$(2) \Rightarrow (1)$
On the contrary, suppose $\mcF(P)$ does not have a linear resolution.
Let $A$ be a multidegree such that for some $i$, $\beta_{i,A}(\mcF(P))$ is nonzero and does not lie on the first linear strand of the minimal free resolution of $\mcF(P)$.
Since $A$ is not on the first linear strand, by Proposition \ref{pro-firststrand}, the induced subgraph on $A$ is not a full bipartite graph.
Let $A_1$ and $A_2$ be the set of elements of rank 1 and 2 in $A$ respectively.
Let $Y_1(A)$ be the simplicial complex defined in Construction \ref{con-Y} on $A_1$.
The complex $Y_1(A)$ is not the irrelevant complex $\{\emptyset\}$ and 
it has at least 2 vertices, since otherwise it would be contractible which is a contradiction.
Let $F=\{p_1,\ldots,p_s\} \subseteq A_1$ be a minimal nonface of $Y_1(A)$ with $s\geq 2$.
There exists an element $q_1\in A_2$ which is connected to $p_1$ but it is not connected to the rest of elements of $F$. Similarly, there exist an element connected to $p_2$ which is not connected to the other elements of $F$.
The induced subposet of $P$ on $\{p_1,p_2,q_1,q_2\}$ is of the form above which contradicts our assumption.
Therefore such a multidegree does not appear in the resolution of $\mcF(P)$. This shows that $\mcF(P)$ has a linear resolution.

\medskip
$(1) \Leftrightarrow (3)$ This is \cite[Theorem 4.2]{Corso-Nagel-01}.
\end{proof}

The condition (2) of theorem above gives a nontrivial characterization of Ferrers graphs.
%Which I do not know if it is already known.
This description for Ferrers graphs is essentially used in characterization of flag ideals with linear resolutions below.

\begin{theorem}
\label{thm-minres}
Let $\mcF(P)$ be the flag ideal of a poset $P$ of rank $r$.
The ideal $\mcF(P)$ has a linear resolution if and only if for $i=1,\ldots,r-1$, the Hasse diagram of $P_{i,i+1}$ is a Ferrers graph.
In particular, if $\mcF(P)$ has a linear resolution then $P$ is connected.
\end{theorem}

\begin{proof}
$(\Leftarrow)$
%Suppose $P$ is ranked  tower product of Ferrers graph.
Let $r = \overline{r}(P)$.
Suppose $\mcF(P)$ does not have a linear resolution.
Let $A$ be a multidegree that appears with a nonzero multigraded Betti number in the free resolution of $\mcF(P)$ but not on the first linear strand.
For $i=1,\ldots,r-1$ let $G_i(A)$ be the bipartite graph on $A_i\cup A_{i+1}$ and let $X_i(A)$ be the simplicial complex whose Stanley-Resiner ideal is the edge ideal of $G_i(A)$.
By Proposition \ref{pro-firststrand}, there exist some $j$ such that $G_j(A)$ is not a full bipartite graph.
By Theorem \ref{thm-towerjoin}, $X(A)$ is homotopy equivalent to $X_1(A)\ast \cdots \ast X_{r-1}(A)$ and it has nontrivial cohomology. By K\"unneth formula $X_j(A)$ also has nonvanishing cohomology and the multidegree $A_j\cup A_{j+1}$ appears in the free resolution of the edge ideal of the bipartite graph $G_j(P)=P_{j,j+1}$ but not on the first linear strand which is a contradiction to our assumption that the edge ideal of all $G_i(P)$ have linear resolutions.

\medskip
$(\Rightarrow)$
Suppose $\mcF(P)$ has a linear resolution. Obviously the poset $P$ must be a pure poset.
%The poset $P$ is the tower product of bipartite graphs $G_i(P)$, $i=1,\ldots,r-1$.
Suppose for some $j$, $G_j(P)$ is not a Ferrers graph.
Then it has an induced subgraph consisting of two separate edges $\{\{p_1,q_1\},\{p_2,q_2\}\}$ with $p_1,p_2\in P_j$ and $q_1,q_2\in P_{j+1}$.
We construct a multidegree $A \subseteq P$ of $\mcF(P)$ that does not lie on the first linear strand as follows. Let $A_j = \{p_1,p_2\}$ and $A_{j+1} = \{q_1,q_2\}$.

i. if $p_1$ and $p_2$ have a parent say $p^{j-1}$ in common then let $A_{j-1} =\{p^{j-1}\}$ and ii. if they do not have a parent in common choose a parent $p_1^{j-1}$ for $p_1$ and a parent $p_2^{j-1}$ for $p_2$ and let $A_{j-1} = \{p_1^{j-1},p_2^{j-1}\}$.
Now we do the same for $A_{j-1}$. In the first case, we choose a parent of $p^{j-1}$ say $p^{j-2}$ and we let $A_{j-2} = \{p^{j-2}\}$. For the second case i. if $p_1^{j-1}$ and $p_2^{j-1}$ have a parent $p^{j-2}$ in common we let $A_{j-2} = \{p^{j-2}\}$ and ii. if they do not have a common parent we choose different parents $p^{j-2}_1,p^{j-2}_2$ and let $A_{j-2} = \{p^{j-2}_1,p^{j-2}_2\}$. We continue this until we define $A_1$.
We use a similar procedure with $q_1$ and $q_2$ and define set $A_{j+2},\ldots,A_{r}$ with the difference that instead of choosing parents we choose children.
Finally we end up with a multidegree $A=\cup_{i=1}^r A_i$ such that for all $i=1,\ldots,r$, the bipartite graphs $G_i(A)$ are isomorphic to one of the graphs below.

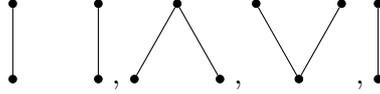
\begin{center}
\medskip
\begin{tikzpicture}[scale=.75, vertices/.style={draw, fill=black, circle, inner sep=1pt}]
              \node [vertices] (0) at (-.75+0,0){};
              \node [vertices] (2) at (-.75+1.5,0){};
              \node [vertices] (1) at (-.75+0,1.33333){};
              \node [vertices] (3) at (-.75+1.5,1.33333){};
      \foreach \to/\from in {0/1, 2/3}
      \draw [-] (\to)--(\from);
      \end{tikzpicture}
,
 \begin{tikzpicture}[scale=.75, vertices/.style={draw, fill=black, circle, inner sep=1pt}]
              \node [vertices] (0) at (-.75+0,0){};
              \node [vertices] (2) at (-.75+1.5,0){};
              \node [vertices] (1) at (-0+0,1.33333){};
      \foreach \to/\from in {0/1, 2/1}
      \draw [-] (\to)--(\from);
      \end{tikzpicture}
,
 \begin{tikzpicture}[scale=.75, vertices/.style={draw, fill=black, circle, inner sep=1pt}]
              \node [vertices] (0) at (-0+0,0){};
              \node [vertices] (1) at (-.75+0,1.33333){};
              \node [vertices] (2) at (-.75+1.5,1.33333){};
      \foreach \to/\from in {0/1, 0/2}
      \draw [-] (\to)--(\from);
      \end{tikzpicture}
,
 \begin{tikzpicture}[scale=.75, vertices/.style={draw, fill=black, circle, inner sep=1pt}]
              \node [vertices] (0) at (-0+0,0){};
              \node [vertices] (1) at (-0+0,1.33333){};
      \foreach \to/\from in {0/1}
      \draw [-] (\to)--(\from);
      \end{tikzpicture}
\end{center}
The simplicial complexes $X_i(A)$ have nonvanishing cohomology and so does their join.
Therefore we get a multidegree that appears in the linear resolution of $\mcF(P)$ but not on the first linear strand by Proposition \ref{pro-firststrand} which is a contradiction. This completes the proof.
\end{proof}

\begin{definition}
A squarefree monomial ideal is called {\it bi-Cohen-Macaulay} if both the ideal and its Alexander dual $I^A$ are Cohen-Macaulay ideals. Equivalently, a squarefree monomial ideal is bi-Cohen-Macaulay if and only if it is Cohen-Macaulay and has a linear resolution. 
\end{definition}

For finite posets $P$ and $Q$, let $\Hom(P,Q)$ be the set of isotone maps $\phi:P\to Q$.
One can see that $\Hom(P,Q)$ is itself a poset with the partial relation
\[
\phi \leq \psi \text{ if and only if } \forall p\in P, \phi(p)\leq \psi(p).
\]
The following propositions shows that for $r,t\in \NN$, the letterplace ideals $L(r,t)$ are the only bi-Cohen-Macaulay flag ideals.

\begin{proposition}
Let $P$ be a finite poset. The flag ideal $\mcF(P)$ is bi-Cohen-Macaulay if and only if $P$ is isomorphic to $\Hom([r],[t])$ for some $r,t\in \NN$.
In other words, after a relevant labeling of elements of $P$, one has $\mcF(P) = L(r,t)$.
\end{proposition}

\begin{proof}
Let $R=\kk[x_1,\ldots,x_t]$ be a polynomial ring in $t$ variables and let $\mathfrak{m} = <x_1,\ldots,x_t>$ be the homogeneous maximal ideal of $R$.
The letterplace ideal $L(r,t)$ is the polarization of the $r$-th power of $\mathfrak{m}$. Hence it is bi-Cohen-Macaulay.

Conversely, suppose $\mcF(P)$ is bi-Cohen-Macaulay.
The ideal $\mcF(P)$ is Cohen-Macaulay and has a linear resolution.
Obviously, $P$ must be a pure poset.
Let $r=\overline{r}(P)$ and $t=|P_1|$.
Since it is Cohen-Macaulay its elements has a labeling that satisfies the conditions in Theorem \ref{thm-mainCM}.
Consider the edges $\{a^i_j, a^{i+1}_j\}$ and $\{a^i_{j'},a^{i+1}_{j'}\}$ with $j\leq j'$ of the bipartite graph $P_{i,i+1}$ for $1\leq i\leq r-1$.
By Theorem \ref{thm-Ferrerschar}, these edges can not be disjoint.
Therefore, $P_{i,i+1}$ contains the edge $\{a^i_j,a^{i+1}_{j'}\}$ or the edge $\{a^i_{j'},a^{i+1}_j\}$ or both.
Since $\mcF(P)$ is Cohen-Macaulay, it can only have $\{a^i_j,a^{i+1}_{j'}\}$.
This implies that $P_{i,i+1} \cong \Hom([2], [t])$ and also $P\cong \Hom([r],[t])$.
\end{proof}

\bibliographystyle{alpha}
\bibliography{MyBib}
\end{document}